\documentclass[11 pt]{amsart}
\usepackage{graphicx}
\usepackage[hidelinks]{hyperref}
\usepackage{color}
\usepackage{amssymb}
\usepackage{epstopdf}
\usepackage{nicefrac}
\usepackage{amsthm}
\usepackage{amsthm}
\usepackage{tikz}
\usepackage{tikz,fullpage}
\usetikzlibrary{patterns}
\numberwithin{equation}{section}
\usepackage[margin=1.27in]{geometry}
\usetikzlibrary{arrows,%
                petri,%
                topaths}%
\usepackage{tkz-berge}
\usepackage[position=top]{subfig}
\usetikzlibrary{%
  matrix,%
  calc,%
  arrows%
}
\newtheorem{thm}{THEOREM}[section]
\newtheorem{conjecture}[thm]{CONJECTURE}

\newtheorem{lemma}[thm]{LEMMA}

\newcommand{\G}{\Gamma}

\begin{document}
\title{Hochschild and cyclic homology of the crossed product of algebraic irrational rotational algebra by finite subgraoups of $SL(2,\mathbb Z)$.}

\author{Safdar Quddus}
\address{Department of Mathematics, Islamic University, Al-Madinah, KSA}
\email{safdar.quddus@go.wustl.edu}

\date{\today}
 
\let\thefootnote\relax

\maketitle

\begin{abstract}

Let $\G \subset SL(2, \mathbb Z)$ be a finite subgroup acting on the irrational rotational algebra $\mathcal A_\theta$ via the restriction of the canonical action of $SL(2,\mathbb Z)$. Consider the crossed product algebra $\mathcal A_\theta^{alg} \rtimes \G$ obtained by the restriction of the $\G$ action on the algberaic irrational rotational algebra. In this paper we prove many results on the homology group of the crossed product algebra $\mathcal A_\theta^{alg} \rtimes \G$. We also analyse the case of the smooth crossed product algebra, $\mathcal A_\theta \rtimes \G$ and calculate some of its homology groups.
\end{abstract}

\bigskip

\section{Introduction and Statement} \label{sec-bck}
For $\theta \in \mathbb R$ let $\mathcal A_\theta$ be the rotation algebra, which is the universal C*-algebra generated by unitaries $U_1$ and $U_2$ satisfying $U_2 U_1 = \lambda U_1 U_2$, where $\lambda =e^{2 \pi i \theta}$. The structures of non-commutative torus has been revealed in a classical paper of Connes [C], the Hochschild and cyclic cohomology are calculated therein. Paper also calculates the Chern-Connes pairing for the smooth non-commutative torus. [ELPH] consider the action of finite subgroups, $\G \subset SL(2,\mathbb Z)$ on the noncommutative torus algebra and they calculated the $K_0$ group of the corresponding crossed product algebra, $K_0(\mathcal A_\theta \rtimes \G)$. One can consider the subalgebra consisting of finitely supported elements, $\mathcal A_\theta^{alg}$. The Hochschild homology of the crossed product algbera, $\mathcal A_\theta^{alg} \rtimes \mathbb Z_2$ was calculated in [O] by Oblomkov and in [B], Baudry described the groups $HH_0(\mathcal A_\theta^{alg} \rtimes \G)$ for all finite subgroups $\G$ of $SL(2,\mathbb Z)$.\par
In the paper [BRT], authors studied the Picard group of the algebraic non-commutative torus algebra, $\mathcal A_\theta^{alg}$. In this paper we extend the results of [O] and [B] to give a complete description of the homology of $\mathcal A_\theta^{alg} \rtimes \G$ for all finite subgroups $\G$ of $SL(2,\mathbb Z)$. We also calculate the periodic cyclic homology of $\mathcal A_\theta^{alg} \rtimes \G$. We use a completely new method to prove our results than in [O] or [B]. We also state in this paper some homology groups of the orbifold $\mathcal A_\theta \rtimes \G$. We end this article with a conjecture on the cyclic homology groups of $\mathcal A_\theta \rtimes \G$.\par
\section{$SL(2, \mathbb Z) $ action on $\mathcal A_\theta$}
The C*-dynamical system $(\mathcal A_\theta, \G, \rho)$ can be obtained for a finite subgroup $\G\subset SL(2,\mathbb Z)$ acting on $\mathcal A_\theta$ in the following way. An element 
\begin{center}
$g= \left[
 \begin{array}{cc}
   g_{1,1} & g_{1,2} \\
   g_{2,1} & g_{2,2}
 \end{array} \right]\in SL(2,\mathbb Z)$
\end{center} acts on the generators $U_1$ and $U_2$ as described below
\begin{center}
$\rho_gU_1=e^{(\pi i g_{1,1} g_{2,1})\theta}U_1^{g_{1,1}}U_2^{g_{2,1}} \text{ and }\rho_gU_2=e^{(\pi i g_{1,2} g_{2,2})\theta}U_1^{g_{1,2}}U_2^{g_{2,2}}$\\.
\end{center}
We leave it to the readers to check that the restriction of the formula described above to the finite subgroups, $\mathbb Z_2, \mathbb Z_3, \mathbb Z_4, \text{and } \mathbb Z_6$ of $SL(2,\mathbb Z)$ is indeed an automorphism of the algebra $\mathcal A_\theta$. We illustrate this action by considering the action of $\mathbb Z_4$ on the noncommutative torus algebra, $\mathcal A_\theta^{alg}$. We notice that the generator of $\mathbb Z_4$ in $SL(2,\mathbb Z)$, $g= \left[
 \begin{array}{cc}
   0 & -1 \\
   1 & 0
 \end{array}
\right]$ acts on $\mathcal A_\theta^{alg}$ as follows
\begin{center}
$ \left[
 \begin{array}{cc}
   0 & -1 \\
   1 & 0
 \end{array}
\right] U_1 = U_2^{-1}$ and  $\left[
 \begin{array}{cc}
   0 & -1 \\
   1 & 0
 \end{array}
\right] U_2 = U_1$.
\end{center}
While $\mathcal A_\theta \rtimes \G$ is defined to the crossed product algebra associated to the C*-dynamical system $(\mathcal A_\theta, \G, \rho)$; $\mathcal A_\theta^{alg} \rtimes \G$ is defined to the associated crossed product algebra for the system $(\mathcal A_\theta^{alg},\G,\rho)$ \emph{without} its completion.

The main results of this paper are the following theorems.
\begin{thm} If $\theta \notin \mathbb Q$, the Hochschild homology groups are as follows:
\begin{center}
$H_0(\mathcal A_{\theta}^{alg} \rtimes \G,\mathcal A_{\theta}^{alg} \rtimes \G ) \cong\begin{cases}
\mathbb C^5 & \text{ for } \G = \mathbb Z_2\\
\mathbb C^7 & \text{ for } \G = \mathbb Z_3\\
\mathbb C^8  & \text{ for } \G = \mathbb Z_4 \\
\mathbb C^9 & \text{ for } \G = \mathbb Z_6. \end{cases}$\\
$HH_1(\mathcal A_{\theta}^{alg} \rtimes \G) \cong 0 \text{ for all finite subgroups } \G \subset SL(2, \mathbb Z)$. \\
$HH_2(\mathcal A_{\theta}^{alg} \rtimes \G)  \cong \mathbb C \text{ for all finite subgroups } \G  \subset SL(2, \mathbb Z)$. \\
$HH_k(\mathcal A_{\theta}^{alg} \rtimes \G)  \cong 0 \text{ for all } k>3 \text{ and all finite subgroups } \G  \subset SL(2, \mathbb Z)$. \\
\end{center}
\end{thm}
We remark that the result of Theorem 1.1 for $\G =\mathbb Z_2$ is presented in [O] and the group $HH_0(\mathcal A_\theta^{alg} \rtimes \G)$ is calculated in [B]. Theorem 1.1 completes the computation of all Hochschild homology groups for $\mathcal A_\theta^{alg} \rtimes \G$ for all $\G \subset SL(2,\mathbb Z)$.
The periodic cyclic homology groups of $\mathcal A_\theta^{alg} \rtimes \G$ are as follows:
\begin{thm} If $\theta \notin \mathbb Q$,
$HC_{even}(\mathcal A_{\theta}^{alg} \rtimes \G) \cong\begin{cases}
\mathbb C^6   & \text{ for } \G = \mathbb Z_2\\
\mathbb C^8   & \text{ for } \G = \mathbb Z_3\\
\mathbb C^9   & \text{ for } \G = \mathbb Z_4 \\
\mathbb C^{10} & \text{ for } \G = \mathbb Z_6. \end{cases}$\\  The odd homology
$HC_{odd}(\mathcal A_{\theta}^{alg} \rtimes \G)\cong 0 \text{ for all finite subgroups } \G \subset SL(2, \mathbb Z)$. \\
\end{thm}
\begin{thm}
For $\theta \notin \mathbb Q, HH_2(\mathcal A_\theta \rtimes \G) \cong \mathbb C$.
\end{thm}
We also examine the $\G$ invariant cycles in $HH_\bullet(\mathcal A_\theta)$ for $\bullet = 0,1$.
\begin{thm}
For $\theta \notin \mathbb Q$ satisfying Diophantine condition, we have 
\begin{center}
$HH_1(\mathcal A_\theta)^{\G} =0$, and $H_0(\mathcal A_\theta)^{\G} \cong \mathbb C$.
\end{center}
\end{thm}

\section{Strategy of the proof}
We prove the above theorems by decomposing each homology group into several simpler group associated to each group element $g \in \Gamma$. Thereafter we develop a resolution for the algebra $\mathcal A_\theta^{alg}$ and calculate the homologies. We shall do so by using an innovative technique of diagram representation for the chains and cycles. We illustrate this by giving  some examples here.
\par Conside the equation $a^1_{6,1} - a^2_{5,0} = \lambda a^1_{4,1} - {\lambda}^{-5} a^2_{5,2}$.\\
The elements $a^1_{6,1}, a^1_{4,1}, a^2_{5,0}$ and $a^2_{5,2}$ can be considered as point in the lattice plane $\mathbb Z^2$.  We respresent this equation as a diagram on the latttice $\mathbb Z^2$ in the following way.\\
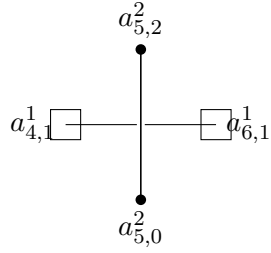
\begin{center}
\begin{tikzpicture}
\draw (-1,0)node(xline)[left] {$a^1_{4,1}$} -- (1,0 )node(xline)[right] {$a^1_{6,1}$};
\draw[draw=white,double=black,very thick] (0,-1)node(yline)[below] {$a^2_{5,0}$} -- (0,1)node(yline)[above] {$a^2_{5,2}$};
\fill (canvas cs:x=0cm,y=1cm) circle (2pt);
\fill (canvas cs:x=0cm,y=-1cm) circle (2pt);
\node at (-1,0) [transition]{};
\node at (1,0) [transition]{};
\end{tikzpicture}
\end{center}
In the proofs we shall use this method to find the homology groups of the product algebars $\mathcal A_\theta^{alg} \rtimes \Gamma$.

\section{Paracyclic modules and spectral decomposition}

One can use the results of [EJ] to deduce the following decomposition of the homology group of the algebra $\mathcal A_\theta^{alg} \rtimes \G$.
\begin{thm}[EJ]
If $\G$ is finite and $|\G|$ is invertible in k, then there is a natural isomorphism of cyclic homology and
\begin{center}
$HH_\bullet(\mathcal A_\theta^{alg} \rtimes \G)= HH_\bullet(H_0(\G, (\mathcal A_\theta^{alg})_\G^\sharp),$
\end{center}
where $(H_0(\G, (\mathcal A_\theta^{alg})_\G^\sharp)$ is the cyclic module
\begin{center}
$H_0(\G,(\mathcal A_\theta^{alg})_\G^\sharp)(n) =  H_0(\G,k[\G] \otimes (\mathcal A_\theta^{alg})^{\otimes {n+1}})$.
\end{center}
\end{thm} 
Since $\G$ is abelian, we can conclude that the group homology $H_0(\G, \mathcal {A}^{alg \natural}_{\theta, \G})$ splits the complex into $|\G|$ disjoint parts.
\begin{center} 
$H_0(\G,\mathcal {A}^{alg \natural}_{\theta,\G})(n) = H_0(\G, k[\G] \otimes (\mathcal {A}_{\theta}^{alg})^{\otimes {n+1}}) = \displaystyle \bigoplus_{t \in \G} ((\mathcal {A}_{\theta, t}^{alg})^{\otimes {n+1}})^{\G}$\\
\end{center}
For each $t \in \G$ ,  the algebra $\mathcal {A}_{\theta, t}^{alg}$ is setwise $\mathcal A_\theta^{alg}$ with the Hochschild differential ${}_{t}b$ in the Hochschild complex $C_\bullet(\mathcal A^{alg}_{\theta , t},\mathcal A^{alg}_{\theta , t})$.
Hence, we can write its Hochschild homology, $HH_\bullet(\mathcal A_\theta^{alg} \rtimes \G)$ as follows
\begin{center}
$HH_\bullet(\mathcal A_\theta^{alg} \rtimes \G)= HH_\bullet(H_0(\G,\mathcal {A}^{alg \natural}_{\theta, \G}))=\displaystyle \bigoplus_{t \in \G} HH_\bullet((\mathcal {A}_{\theta , t}^{alg \bullet})^{\G})$.
\end{center} 
It is enough to calculate $HH_\bullet((\mathcal {A}_{\theta, t}^{alg \bullet})^{\G})$ for each $t \in \G$. To calculate these individual homology groups, we can use the following lemma. Proof of which is an easy exercise.

\begin{lemma}
Let 
\begin{center}
$J_\ast :=  0 \xleftarrow{d_v} A \xleftarrow{d_v} (A^{\otimes2}) \xleftarrow{d_v} (A^{\otimes3}) \xleftarrow{d_v} (A^{\otimes4}) \xleftarrow{d_v} (A^{\otimes5}) \xleftarrow{d_v} ...$
\end{center}
be a chain complex. For a given $\G$ action on $A$, consider the following chain complex, with chain map $d_v^\G :(A^{\otimes n})^\G \rightarrow (A^{\otimes{n-1}})^\G$ induced from the map $d_v : A^{\otimes n} \rightarrow A^{\otimes{n-1}}$.
\begin{center}
$J^\G_{\ast} := 0 \xleftarrow{d_v^\G} A^\G \xleftarrow{d_v} (A^{\otimes2})^\G \xleftarrow{d_v^\G} (A^{\otimes3})^\G \xleftarrow{d_v^\G} (A^{\otimes4})^\G \xleftarrow{d_v^\G} (A^{\otimes5})^\G \xleftarrow{d_v^\G} ...$
\end{center}
With the $\G$ action commuting with the differential $d_v$. We have the following group equality $H_q(J_\ast^\G, d_v^\G)=H_q(J_\ast, d_v)^\G$.
\end{lemma}

\section{Revisiting Connes' resolution}

Firstly we have to ensure if the projective resolution introduced in [C] can be adjusted to the algebraic noncommutative torus algebra:

\begin{lemma} The following is a projective resolution of $\mathcal A_\theta^{alg}$
\begin{center}
$\mathcal A^{alg}_\theta \xleftarrow{\epsilon} \mathfrak B^{alg}_\theta \xleftarrow{b_1} \mathfrak B^{alg}_\theta \displaystyle \bigoplus \mathfrak B^{alg}_\theta \xleftarrow{b_2} \mathfrak B^{alg}_\theta$
\end{center}
where, $\mathfrak B^{alg}_\theta = \mathcal A_\theta^{alg} \otimes (\mathcal A_\theta^{alg})^{op}$ \newline
$\epsilon(a\otimes b ) = ab, \newline
b_1(1\otimes e_j)= 1\otimes {U_j}- {U_j}\otimes 1 \text{ and } \newline
b_2(1\otimes( e_1 \wedge e_2 ) ) = (U_2\otimes 1 - \lambda \otimes U_2 )\otimes e_1- ( \lambda U_1\otimes 1 - 1\otimes U_1 )\otimes e_2.$
\end{lemma}
\begin{proof}Since $\epsilon b_1(1\otimes e_{j})= U_j-U_j=0$, hence we have $im(b_1) \subset ker(\epsilon)$. Let an element $x \in ker(\epsilon)$ be 
\begin{center}
$x=\displaystyle \sum a_{\nu, {\nu}'}X^{\nu}Y^{{\nu}'}$, where $X^{\nu}=U_1^{n_1}U_2^{n_2}\otimes 1, Y^{{\nu}'}=1\otimes U_1^{{n_1}'}U_2^{{n_2}'}.$\\
\end{center}
Hence, $x=\displaystyle \sum a_{\nu, {\nu}'}X^{\nu}(Y^{{\nu}'}-X^{{\nu}'})$. Consider the following equality
\begin{center}
$(1\otimes {U_2}^{n_2})(1\otimes {U_1}^{n_1})-({U_1}^{n_1}\otimes 1)({U_2}^{n_2}\otimes 1) = (1\otimes {U_2}^{n_2})(\displaystyle \sum_{j=0}^{n_{1}-1}U_1^j\otimes U_1^{n_{1}-1-j})(1\otimes U_{1}-U_{1}\otimes 1) \newline +({U_1}^{n_1}\otimes 1)(\displaystyle \sum_{j=0}^{n_{2}-1}U_2^j\otimes U_2^{n_{2}-1-j})(1\otimes U_{2}-U_{2}\otimes 1)$.
\end{center}
Since $x=\displaystyle \sum a_{\nu, {\nu}'}X^{\nu}(Y^{{\nu}'}-X^{{\nu}'})$ can be written in terms of the image of the right hand side coefficients in above equation. Hence $ker(\epsilon) \text{ is generated by } (1\otimes U_{1}-U_{1}\otimes 1) \text{ and }(1\otimes U_{2}-U_{2}\otimes 1).$

\par Given an element $x= x_1\otimes e_1 - x_2\otimes e_2$ where, $x_1 , x_2 \in \mathcal{A}_\theta^{alg} \in ker(b_1)$, we have, 
\begin{center}
$x_1(1\otimes U_{1}-U_{1}\otimes 1)=x_2(1\otimes U_{2}-U_{2}\otimes 1)$.\\\
\end{center}
To prove $x \in Im b_2$, it is enough to find $y \in \mathcal A^{alg}_\theta \otimes^{op} \mathcal A_\theta^{alg}$ such that $x_1=y(U_{2}\otimes 1 - \lambda\otimes U_{2})$. Let $Z=(\lambda U_2^{-1}\otimes U_2)$ then we have the following 
\begin{center}
$x_1(\displaystyle \sum_{-\infty}^{\infty}Z^k)=0$.
\end{center}
Also one can calculate that 
\begin{center}
$x_1(\displaystyle \sum_{-\infty}^{\infty}Z^k)(1\otimes U_{1}-U_{1}\otimes 1)=x_1(1\otimes U_{1}-U_{1}\otimes 1)\displaystyle \sum_{-\infty}^{\infty}(U_2^{-1}\otimes U_2)^k
= x_2(1\otimes U_{2}-U_{2}\otimes 1)\displaystyle \sum_{-\infty}^{\infty}(U_2^{-1}\otimes U_2)^k =0 $.\\
\end{center}
After writing $x_1=\displaystyle \sum a_k Z^k$, where $(a_k)$ is a finitely supported over the elements of the closed subalgebra generated by $U_1\otimes 1, U_2\otimes 1, 1\otimes U_1$. We have 
\begin{center}
$x_1=\displaystyle \sum a_k(Z^k - 1) = \displaystyle \sum a_k(\displaystyle \sum_{0}^{k-1}Z^k)(Z-1)$.
\end{center}
It can be observed that for $y:=\displaystyle \sum b_k Z^k$ with $b_k=\displaystyle \sum_{j-1=k}^{\infty} a_j$ satisfies, $x_1=y(U_{2}\otimes 1 - \lambda\otimes U_{2})$. 
\par
Hence $y \in \mathcal A_\theta^{alg}$.
\end{proof}

\section{$\G$ invariant Hochschild homology group, $HH_\bullet((\mathcal A_{\theta,1}^{alg})^\G)$}
In this section we compute the Hochschild homology for the $g=1$ part appearing as one of the decomposed cases while we compute the Hochschild homology for the crossed product algebras $\mathcal A_\theta^{alg} \rtimes \G$. We calculate these groups for each of the four cases, corresponding to the groups $\mathbb Z_2, \mathbb Z_3, \mathbb Z_4, \mathbb Z_6$. \par
We notice that from Lemma 3.2 we can simplify the group as follows
\begin{center}
$HH_\bullet((\mathcal A_{\theta,1}^{alg})^\G)=(HH_\bullet(\mathcal A_{\theta,1}^{alg}))^\G$
\end{center}
BY the Kozul complex in Lemma 4.1, the Hochschild homology of $\mathcal A_\theta^{alg}$ is computed by the following complex
\begin{center}
 $0 \leftarrow \mathcal A^{alg}_\theta \xleftarrow{1 \otimes _{\mathfrak B^{alg}_\theta}b_1} \mathcal A^{alg}_\theta \displaystyle \bigoplus \mathcal A^{alg}_\theta \xleftarrow{1 \otimes_{\mathfrak B^{alg}_\theta}b_2} \mathcal A^{alg}_\theta$
\end{center}
We have 
\begin{center}
$H_2(\mathcal A_\theta^{alg}, \mathcal A_\theta^{alg}) = ker( 1 \otimes b_2)$,
\end{center}
where,
\begin{center}
 $(1 \otimes b_2)(a\otimes I)=a\otimes_{\mathfrak B_\theta^{alg}}(U_2\otimes I-\lambda\otimes U_2)\otimes e_1-a\otimes_{\mathfrak B_\theta^{alg}}(\lambda U_1\otimes I-I\otimes U_1)\otimes e_2= (aU_2-\lambda  U_2a,U_1 a-\lambda a U_1)$ .
\end{center}
\par It is clear that $H_2(\mathcal A_\theta^{alg}, \mathcal A_\theta^{alg})= \langle \varphi \in \mathcal A_\theta^{alg} | \varphi_{n,m-1}={\lambda}^{n+1}\varphi_{n,m-1}={\lambda}^{m}\varphi_{n,m-1} \rangle$. To solve the above condition on a cycle in $H_2(\mathcal A_\theta^{alg}, \mathcal A_\theta^{alg})$, we observe that we need to have $m=0$ for the second equality to hold and also we need that $n=-1$ for the first equality. Hence we see that $\varphi_{-1,-1}U_1^{-1}U_2^{-1}$, the elements of $\mathcal A_\theta^{alg}$ supported at $U_1^{-1}U_2^{-1}$ is the required generator for $H_2(\mathcal A_\theta^{alg}, \mathcal A_\theta^{alg})$.  \newline
To calculate the $\G$ invariant algebra we need to push this cycle $\varphi_{-1,-1}U_1^{-1}U_2^{-1}$ into the bar complex $C_\ast$ using the map $h_2: J_2(\mathcal A_\theta^{alg}) \to C_2(\mathcal A_\theta^{alg})$, where
\begin{center}
$h_2(1 \otimes (e_1 \wedge e_2)) = I \otimes U_2 \otimes U_1 - \lambda \otimes U_1 \otimes U_2$.
\end{center}
After considering the action of $t \in \G$ on the element $(1 \otimes h_2)(\varphi_{-1,-1}U_1^{-1}U_2^{-1})$. Then we pull back the element $t \cdot ((1 \otimes h_2)( \varphi_{-1,-1}U_1^{-1}U_2^{-1}))$ to the Kozul complex to check its invariance. Using this technique we have the following computations.
\subsection{$H_2(\mathcal A_\theta^{alg}, \mathcal A_\theta^{alg})^{\G}$}
To calculate $H_2(\mathcal A_\theta^{alg}, \mathcal A_\theta^{alg})^{\G}$ we need to push the 2-cocycles into the bar complex and thereafter pull it back after the action of $\mathbb Z_2$ on the class. Consider the map, $(1 \otimes b_2)$, we see that \newline
$(1\otimes h_2)(\varphi_{-1,-1}U_1^{-1}U_2^{-1})=\varphi_{-1,-1}(1\otimes h_2)(U_1^{-1}U_2^{-1})\newline=a_{-1,-1} (U_1^{-1}U_2^{-1}\otimes U_2 \otimes U_1 - \lambda U_1^{-1}U_2^{-1} \otimes U_1 \otimes U_2)$ \newline
After $t \in \mathbb Z_2$ acts on the element $a_{-1,-1} U_1^{-1}U_2^{-1}\otimes U_2 \otimes U_1 - \lambda U_1^{-1}U_2^{-1} \otimes U_1 \otimes U_2$ we get 
\begin{center}
$a_{-1,-1} (U_1U_2\otimes U_2^{-1} \otimes U_1^{-1} - \lambda U_1U_2 \otimes U_1^{-1} \otimes U_2^{-1})$
\end{center}
Now, we consider the element $(1 \otimes k_2)(a_{-1,-1} (U_1U_2\otimes U_2^{-1} \otimes U_1^{-1} - \lambda U_1U_2 \otimes U_1^{-1} \otimes U_2^{-1}))$, where for $\nu = (n_1,n_2)$ and $\mu = (m_1,m_2)$ the chain map $k_2: C_2(\mathcal A_\theta^{alg}) \to J_2(\mathcal A_\theta^{alg})$ is given by the following formula
\begin{center}
$k_2(I \otimes U^\nu \otimes U^\mu) = (U_1\otimes I)^{n_1}  \displaystyle \frac{{\lambda}^{n_2 m_1} ( U_1 \otimes I)^{m_1} - {\lambda}^{m_1 m_2}(I \otimes U_1)^{m_1}}{{\lambda}^{n_2}(U_1 \otimes I)-{\lambda}^{-m_2}(I \otimes U_1)}   \newline \displaystyle \frac{( U_2 \otimes I)^{n_2} -{\lambda}^{n_2}(I \otimes U_2)^{n_2}}{(U_2 \otimes I)-\lambda(I \otimes U_2 )} (I \otimes U_2)^{m_2} \otimes e_1 \wedge e_2.$
\end{center}
After $t \in \mathbb Z_2$ acts on $(1\otimes h_2)(\varphi_{-1,-1}U_1^{-1}U_2^{-1})$ , we consider the pull back of $ -1 \cdot (1\otimes h_2)(\varphi_{-1,-1}U_1^{-1}U_2^{-1})$ on the Kozul complex.
\begin{center}
$(1 \otimes k_2)a_{-1,-1} (U_1U_2\otimes U_2^{-1} \otimes U_1^{-1} - \lambda U_2U_1 \otimes U_1^{-1} \otimes U_2^{-1})\newline = a_{-1,-1}U_1U_2\displaystyle \frac{\lambda  (U_1 \otimes 1)^{-1}-(I \otimes U_1)^{-1}}{{\lambda}^{-1}(U_1 \otimes I) -(I \otimes U_1)} \displaystyle \frac{(U_2 \otimes I)^{-1}-{\lambda}^{-1}(I \otimes U_2)^{-1}}{(U_2 \otimes I) -{\lambda}(I \otimes U_2)}-0$.
\end{center}
To simply the relations above we notice that
\begin{center}
$\displaystyle \frac{\lambda  (U_1 \otimes 1)^{-1}-(I \otimes U_1)^{-1}}{{\lambda}^{-1}(U_1 \otimes I) -(I \otimes U_1)} = -\lambda (U_1^{-1} \otimes U_1^{-1})$ and 
\end{center}
\begin{center}
$\displaystyle \frac{(U_2 \otimes I)^{-1}-{\lambda}^{-1}(I \otimes U_2)^{-1}}{(U_2 \otimes I) -{\lambda}(I \otimes U_2)}= -{\lambda}^{-1}(U_2^{-1} \otimes U_2^{-1})$.
\end{center}
Hence we have \newline
$(1 \otimes k_2)a_{-1,-1} (U_1U_2\otimes U_2^{-1} \otimes U_1^{-1} - \lambda U_2U_1 \otimes U_1^{-1} \otimes U_2^{-1})\newline=  a_{-1,-1}U_1U_2  \cdot -\lambda (U_1^{-1} \otimes U_1^{-1}) \cdot  -{\lambda}^{-1}(U_2^{-1} \otimes U_2^{-1}) \newline=
a_{-1,-1}U_2 U_1^{-1} \cdot (U_2^{-1} \otimes U_2^{-1})=a_{-1,-1}U_2^{-1}U_2 U_1^{-1}U_2^{-1}=a_{-1,-1}U_1^{-1}U_2^{-1}$.

Hence we have, $(H_2(\mathcal A_\theta^{alg}, \mathcal A_\theta^{alg}))^{\mathbb Z_2} \cong \mathbb C$.

Similarly, we can compute the groups $(H_2(\mathcal A_\theta^{alg}, \mathcal A_\theta^{alg}))^{\G}$ for $\G = \mathbb Z_3, \mathbb Z_4, \mathbb Z_6$ and find them to be  isomorphic to $\mathbb C$.

\subsection{$H_1(\mathcal A_\theta^{alg}, \mathcal A_\theta^{alg})^\G$}
We compute this group in a similar way to our calculation of the group $(H_2(\mathcal A_\theta^{alg}, \mathcal A_\theta^{alg}))^{\G}$. We notice that we have the group $H_1(\mathcal A_\theta^{alg}, \mathcal A_\theta^{alg})$ equals
\begin{center}
$ker(1 \otimes b_1) / im( 1\otimes b_2)$.
\end{center}
$(1\otimes b_1)(a\otimes I\otimes e_j)=a\otimes_{\mathfrak B_\theta^{alg}}(I\otimes U_j-U_j\otimes I)= U_ja-aU_j$, Hence, \newline
$ker( 1\otimes b_1) = \langle (\varphi^1, \varphi^2) \in \mathcal A_\theta^{alg} \oplus \mathcal A_\theta^{alg}| U_1\varphi^1-\varphi^1U_1=\varphi^2U_2-U_2\varphi^2 \rangle$. \newline
While we know that
\begin{center}
$im(1 \otimes b_2)= \langle (\varphi^1,\varphi^2)\in  \mathcal A_\theta^{alg} \oplus \mathcal A_\theta^{alg}| \exists \varphi \in \mathcal A_\theta^{alg}; \varphi^1=\varphi U_2-\lambda U_2 \varphi \text{ and } \varphi^2 = U_1 \varphi -\lambda \varphi U_1 \rangle$ 
\end{center}
Hence we see that $H_1(\mathcal A_\theta^{alg} ,\mathcal A_\theta^{alg}) \cong \mathbb C^2$, and is generated by the equivalence classes of $\bar{\varphi^1_{-1,0}}$ and $\bar{\varphi^2_{0,-1}}$.\par
We need to check the invariance of $H_1(\mathcal A_\theta^{alg}, \mathcal A_\theta^{alg})$ under each of the finite subgroups to get the desired group $(H_1(\mathcal A_\theta^{alg}, \mathcal A_\theta^{alg}))^\G$.
\begin{thm}
$H_1(\mathcal A_\theta^{alg}, \mathcal A_\theta^{alg})^{\G} = 0$.
\end{thm}
\begin{proof}
For $\G = \mathbb Z_2$, to check the invariance we need to push the cycle to the bar complex, $C_\ast(\mathcal A_\theta^{alg})$ and then consider the natural action that exists on the bar complex. \newline
Using  the map $h_1: J_1(\mathcal A_\theta^{alg}) \to C_1(\mathcal A_\theta^{alg})$, 
\begin{center}
$h_1(I \otimes e_i) =I \otimes U_j$
\end{center}
we obtain $( 1\otimes h_1)(\varphi^1_{-1,0}U_1^{-1},0)= \varphi^1_{-1,0} U_1^{-1} \otimes U_1$, thereafter we consider the action of $ -1 \in \mathbb Z_2$. Hence we obtain the element
\begin{center}
$\varphi^1_{-1,0}U_1 \otimes U_1^{-1}$.
\end{center}
We now pull this transformed element on back to the Kozul complex for its comparison with the original cycle using the map 
\begin{center}
$k_1(I \otimes U^\nu) = A_\nu \otimes e_1 + B_\nu \otimes e_2$, where $\nu =(n_1, n_2)$.
\end{center}
Coefficient $A_\nu = (I \otimes U_2)^{n_2}((U_1 \otimes I)^{n_1}-(I \otimes U_1)^{n_1}) ((U_1 \otimes I)-(I \otimes U_1))^{-1}$ while the coefficient \newline $B_\nu = (I \otimes U_1)^{n_1}((U_2 \otimes I)^{n_2}-(I \otimes U_2)^{n_2}) ((U_2 \otimes I)-(I \otimes U_2))^{-1}$. These maps
 induce quasi-isomorphism, whence we have the push forwards and pull backs of cycles as cycles in their corresponding complexes.\par
Proceeding further with a general cycle $\varphi \in H_1(\mathcal A_\theta^{alg}, {}_{-1}\mathcal A_\theta^{alg})$, we can express it in terms of generators $\bar{\varphi^1_{-1,0}}$ and $\bar{\varphi^2_{0,-1}}$ as follows,
\begin{center}
$\varphi = aU_1^{-1}\otimes e_1 + b U_2^{-1} \otimes e_2$.
\end{center}
\begin{center}
$(1 \otimes k_1)(-1 \cdot(1 \otimes h_1)(\varphi)) = (1 \otimes k_1)(-1 \cdot (a U_1^{-1} \otimes U_1 + b U_2 \otimes U_2^{-1}))=(1 \otimes k_1) (a U_1\otimes U_1^{-1} + b U_2 \otimes U_2^{-1})=a U_1 A \otimes e_1 + a U_1 B \otimes e_2 + b U_2 A' \otimes e_1 + b U_2 B' \otimes e_2=(a U_1 A + b U_2 A') \otimes e_1 + (a U_1 B + b U_2 B') \otimes e_2$.
\end{center}
Let us calculate $U_1 A \in \mathcal A_\theta^{alg}$. We know from the formula described above that
\begin{center}
$A = \displaystyle \frac{(U_1 \otimes I)^{-1} - (I \otimes U_1)^{-1}}{(U_1 \otimes I) -(I \otimes U_1)}$.
\end{center}
We can simplify the formula for $A$, 
\begin{center}
$A= \displaystyle \frac{(U_1 \otimes I)^{-1} - (I \otimes U_1)^{-1}}{(U_1 \otimes I) -(I \otimes U_1)} = - U_1^{-1} \otimes U_1^{-1}$.
\end{center}
Hence using $\mathcal A_\theta^{alg}$ as a $\mathfrak B_\theta^{alg}$ module, we have $U_1 A = -U_1^{-1}$. Similarly we compute that $U_2 B'=-U_2^{-1}$. As for $A',B$; we see that from the formula above that $A' =B=0$, hence we have 
\begin{center}
$(1 \otimes k_1)(-1 \cdot(1 \otimes h_1)(aU_1^{-1}\otimes e_1 + b U_2^{-1} \otimes e_2)) =-(aU_1^{-1}\otimes e_1 + b U_2^{-1} \otimes e_2)$.
\end{center}
Whence,
\begin{center}
$(H_1(\mathcal A_\theta^{alg} , \mathcal A_\theta^{alg}))^{\mathbb Z_2} = 0$.
\end{center}
Through similar computation for $\G=\mathbb Z_3, \mathbb Z_4, \mathbb Z_6$, we get that, 
\begin{center}
$(H_1(\mathcal A_\theta^{alg} , \mathcal A_\theta^{alg}))^{\G} = 0$.
\end{center}
\end{proof}
\subsection{$H_0(\mathcal A_\theta^{alg}, \mathcal A_\theta^{alg})^\G$}
For the zeroth homology we have a simple calculation, through the fact that $k_0 =h_0 =id$. Hence the natural action on the bar complex translates into the Kozul complex without any changes to it. \par
We know that
\begin{center}
$H_0(\mathcal A_\theta^{alg}, \mathcal A_\theta^{alg}) = \mathcal A_\theta^{alg} /  \langle im(1 \otimes b_1) \rangle$
\end{center}
\begin{center}
$im(1 \otimes b_1)= \langle \varphi \in \mathcal A_\theta^{alg} | \varphi = U_1 \varphi^1 - \varphi^1 U_1 + U_2 \varphi^2 - \varphi^2 U_2 \rangle$
\end{center}
Since $U_1U_2 \in im(1 \otimes b_1)$ we have the group $H_0(\mathcal A_\theta^{alg} , \mathcal A_\theta^{alg})= \langle \bar{\varphi_{0,0}} \rangle$. It is clearly invariant under the $\mathbb Z_2$ action. Hence we have the following result
\begin{center}
$H_0(\mathcal A_\theta^{alg} , \mathcal A_\theta^{alg})^{\mathbb Z_2} \cong \mathbb C$.
\end{center}
In general we have that 
\begin{center}
$H_0(\mathcal A_\theta^{alg} , \mathcal A_\theta^{alg})^{\G} \cong \mathbb C$.
\end{center}
for $\G = \mathbb Z_2, \mathbb Z_3, \mathbb Z_4, \mathbb Z_6$.
\section{$\mathbb Z_2$ action on $\mathcal A_\theta^{alg}$}
\subsection{Hochschild homology}

\begin{thm}
$HH_0((\mathcal {A}_{\theta, -1}^{alg \bullet})^{\mathbb {Z}_2}) \cong \mathbb C^4$, while $HH_k((\mathcal {A}_{\theta, -1}^{alg \bullet})^{\mathbb {Z}_2})$ is a trivial group for all $k \geq 1$
\end{thm}

\begin{lemma}
Consider the following chain complex $J_{\ast, -1}^{\mathbb Z_2}$
\begin{center}
$J_{\ast , -1}^{\mathbb Z_2} := 0 \xleftarrow{{}_{-1}b} (\mathcal A^{alg}_{\theta, -1})^{\mathbb Z_2} \xleftarrow{{}_{-1}b} ((\mathcal A^{alg}_{\theta, -1})^{\otimes 2})^{\mathbb Z_2} \xleftarrow{{}_{-1}b} ((\mathcal A^{alg}_{\theta, -1})^{\otimes 3})^{\mathbb Z_2}...$\\
\end{center}
where,
\begin{center}
${}_{-1}b(a_0\otimes a_1\otimes....\otimes a_n)= b'(a_0\otimes a_1\otimes....\otimes a_n)+(-1)^n((-1\cdot a_n)a_0\otimes a_1\otimes....\otimes a_{n-1}.)$\\
\end{center}
Then,
\begin{center}
$H_\bullet(J_{\ast , -1}^{\mathbb Z_2}, {}_{-1}b) = (H_\bullet(J_\ast(\mathcal A_\theta^{alg}, {}_{-1}\mathcal A_\theta^{alg})), b)^{\mathbb Z_2}$.
\end{center}
\end{lemma}
\begin{proof}
By considering ${}_{-1}\mathcal A_\theta^{alg}$ as a twisted $\mathcal A_\theta^{alg}$ module, we can integrate the twisted part $((-1\cdot a_n)a_0\otimes a_1\otimes....\otimes a_{n-1})$ in the module structure of ${}_{-1} \mathcal A_\theta^{alg}$. With this we have the following equality
\begin{center}
$H_\bullet(J_{\ast , -1}^{\mathbb Z_2}, {}_{-1}b) = H_\bullet((J_\ast(\mathcal A_\theta^{alg}, {}_{-1}\mathcal A_\theta^{alg})^{\mathbb Z_2}, b)$.
\end{center}

The map ${}_{-1}b$ has been modified into the regular Hochschild map, $b$ of a Hochschild complex $C(\mathcal A_\theta^{alg}, {}_{-1}\mathcal A_\theta^{alg})$, with twisted bimodule structure of ${}_{-1}\mathcal A_\theta$ that is as below
\begin{center}
$a{\alpha}=(-1\cdot a)\alpha \text{ and } \alpha a=\alpha a; a\in \mathcal A_\theta, \alpha \in {}_{-1}\mathcal A_\theta$.
\end{center}
Using Lemma 3.2 we can simplify $H_\bullet((J_\ast(\mathcal A_\theta^{alg}, {}_{-1}\mathcal A_\theta^{alg})^{\mathbb Z_2}, b)$.
\begin{center}
$H_\bullet((J_\ast(\mathcal A_\theta^{alg}, {}_{-1}\mathcal A_\theta^{alg})^{\mathbb Z_2}, b)=(H_\bullet(J_\ast(\mathcal A_\theta^{alg}, {}_{-1}\mathcal A_\theta^{alg}), b))^{\mathbb Z_2}$
\end{center}
\end{proof}

Hence using the adjusted Connes' complex for the algebraic case we can now calculate the homology groups.
\begin{center}
\begin{enumerate}
\item[$\bullet$]$H_0(\mathcal A_\theta^{alg},{}_{-1}\mathcal A^{alg}_\theta)= {}_{-1}\mathcal A^{alg}_\theta\otimes_{\mathfrak B_\theta^{alg}} \mathfrak B_\theta^{alg} / Image(1\otimes b_1),$
\item[$\bullet$]$H_1(\mathcal A_\theta^{alg},{}_{-1}\mathcal A^{alg}_\theta)= Ker(1\otimes b_1) / Image(1\otimes b_2),$
\item[$\bullet$]$H_2(\mathcal A_\theta^{alg},{}_{-1}\mathcal A^{alg}_\theta)= Ker(1\otimes b_2)$.
\end{enumerate}
\end{center}

\begin{lemma}$H_2(\mathcal A_\theta^{alg},{}_{-1}\mathcal A^{alg}_\theta) \cong 0.$
\end{lemma}
\begin{proof}
We consider the map $(1 \otimes b_2)$ in the tensor complex. To calculate the kernel of this map we have a closer look at the following map,
\begin{center}
$(1\otimes b_2)(a\otimes I)=a\otimes_{\mathfrak B_\theta^{alg}}(U_2\otimes I-\lambda\otimes U_2)\otimes e_1-a\otimes_{\mathfrak B_\theta^{alg}}(\lambda U_1\otimes I-I\otimes U_1)\otimes e_2$.
\end{center}
Using the twisted bimodule structure of ${}_{-1}\mathcal A_\theta^{alg}$ over $\mathcal A_\theta^{alg}$, the equation can be simplified to the following.
\begin{center}
$(1\otimes b_2)(a\otimes I)= (aU-\lambda  U_2^{-1}a,U_1^{-1}a-\lambda a U_1)$.
\end{center}
Hence we obtain the following relation over an element $(a \otimes 1)$ to reside in $ker( 1\otimes b_2)$. 
\begin{center}
$H_2(\mathcal A_\theta^{alg},{}_{-1}\mathcal A^{alg}_\theta)= \left\lbrace a\in {}_{-1}\mathcal A_\theta^{alg} | a_{n,m}={\lambda}^{m-1}a_{n-2,m}; a_{n-1,m}=\lambda^n a_{n-1,m-2}\right\rbrace$ .
\end{center}
Since, no such nontrivial element exists in ${}_{-1}\mathcal A_\theta^{alg}$ because if it did then all $a_{m,n}$ have to be same up to multiple of $\lambda$ but since the algebra consists of finitely supported elements so they are reduced to zero. Therefore, we have the desired result,
\begin{center}
$H_2(\mathcal A_\theta^{alg},{}_{-1}\mathcal A^{alg}_\theta)=0$.
\end{center}
\end{proof}
\begin{lemma}$H_0(\mathcal A_\theta^{alg},{}_{-1}\mathcal A^{alg}_\theta)^{\mathbb Z_2} \cong \mathbb C^4.$
\end{lemma}
\begin{proof}
We have the map $(1 \otimes b_1)$ in the tensor complex defined below.
\begin{center}
$(1\otimes b_1)(a\otimes I\otimes e_j)=a\otimes_{\mathfrak B_\theta^{alg}}(I\otimes U_j-U_j\otimes I)= U_j^{-1}a-aU_j$.
\end{center}
As before if we use the twisted bicomplex structure of ${}_{-1}\mathcal A_\theta^{alg}$ over the algebra $\mathcal A_\theta^{alg}$, the map$(1 \otimes b_1)$ can be simplified as follows,
\begin{center}
$b_1(a_1,0)= a_1U_1-U_1^{-1}a_1 \text{ and } b_1(0,a_2)=a_2U_2- U_2^{-1}a_2$.
\end{center}
Observe that  $b_1(a_1,a_2)= b_1(a_1,0)+b_1(0,a_2)$. Further we can simplify the calculation by considering only elements of the type $b_1(U_1^nU_2^m,0)$, and similarly for $b_1(0,a_2)$ we can consider the elements of the type $b_1(0,U_1^nU_2^m)$.
We observe that 
\begin{center}
$b_1(U_1^nU_2^m,0)= U_1^{-1}({U_1}^n{U_2}^m)-({U_1}^n{U_2}^m)U_1=U_1^{n-1}U_2^m(1-{\lambda}^{-m}U_1^2)$,
\end{center}
and
\begin{center}
$b_1(0,U_1^nU_2^{m-1})= U_2^{-1}({U_1}^n{U_2}^{m-1})-({U_1}^n{U_2}^m)= {\lambda}^{-1}U_1^nU_2^{m-2}-U_1^nU_2^m=({\lambda}^{-n}-U_2^2)U_1^nU_2^{m-2}$.
\end{center}
Let us consider the case $m=n=1$. In this case we can see that the elements $U_j$ and $U_j^{-1}$ are linearly related in $H_0(\mathcal A_\theta^{alg}, {}_{-1}\mathcal A_\theta^{alg})$.
\begin{center}
$b_1(U_1U_2,0)= U_2(1-\lambda^{-1}U_1^2)$ and  $b_1(0,U_1U_2)=(\lambda^{-1}-U_2^2)U_1$.\\
\end{center}
Hence we can conclude that 
\begin{center}
$H_0(\mathcal A_\theta^{alg},{}_{-1}\mathcal A^{alg}_\theta)= {}_{-1}\mathcal A^{alg}_\theta / \langle(U_1^{n-1}U_2^m(1-{\lambda}^{-m}U_1^2),({\lambda}^{-n}-U_2^2)U_1^nU_2^{m-2} \rangle $.
\end{center}
Now to obtain $(H_0(\mathcal A_\theta^{alg}, {}_{-1}\mathcal A_\theta^{alg}))^{\mathbb Z_2}$  we consider complex map $h: J_\ast \to C_\ast$, since the map $h_0 : J_0(\mathcal A_\theta^{alg}) \to C_0(\mathcal A_\theta^{alg})$ is the identity map, the $\mathbb Z_2$ action on the bar complex is translated to the Kozul complex with no alteration. Hence we get that
\begin{center}
$HH_0((\mathcal {A}_{\theta, -1}^{alg \bullet})^{\mathbb {Z}_2})= \langle \bar{a_{0,0}}, \bar{a_{1,0}}, \bar{a_{0,1}} , \bar{a_{1,1}} \rangle$\\.
\end{center}
\end{proof}

\begin{lemma}$H_1(\mathcal A_\theta^{alg},{}_{-1}\mathcal A^{alg}_\theta) \cong 0$.
\end{lemma}
\begin{proof}
 From the previous calculations we do have an explicit formula for the kernel and the image equations of $H_1(\mathcal A_\theta^{alg}, {}_{-1}\mathcal A_\theta^{alg})$.
\begin{center}
$a^1_{n,m-1}-a^2_{n-1,m-2}={\lambda}^{m-1}a^1_{n-2,m-1}-{\lambda}^{1-n}a^2_{n-1,m}$ ( Kernel condition )
\end{center}
\begin{center}
$a^1_{n,m-1}=a_{n.m-2}-{\lambda}^{1-n}a_{n,m} ; a^2_{n-1,m-1}=a_{n,m-1}-{\lambda}^{m}a_{n-2,m-1}$( Image of $a\in \mathcal A^{alg}_\theta)$
\end{center}
We introduce a combinatorial method of plotting these equations on the $\mathbb Z^2$ plane to study their solutions.
\begin{center}
For two elements $a ,b$ in the lattice plane $\mathbb Z^2$ are said to be \emph{kernel-connected} and drawn like\\
\begin{center}
\begin{tikzpicture}
\draw (-0.2,0) -- (1.2,0) node(xline)[right] {$b$};
\draw (1.2,0) -- (-0.2,0) node(xline)[left] {$a$};
\fill (canvas cs:x=-0.2cm,y=0cm) circle (2pt);
\fill (canvas cs:x=1.2cm,y=0cm) circle (2pt);
\end{tikzpicture}
\end{center}
if there exists a kernel equation containing $a$ and $b$.
\end{center}
For example consider the following kernel equation.
\begin{center}
$a^1_{6,1}-a^2_{5,0}={\lambda}a^1_{4,1}-{\lambda}^{-5}a^2_{5,2}$.\\
\end{center}
Then, we have the diagram below as the corresponding \emph{kernel-diagram}, with the boxes representing elements of $a^1_{\bullet, \bullet}$ and the filled circles are elements of $a^2_{\bullet, \bullet}$.
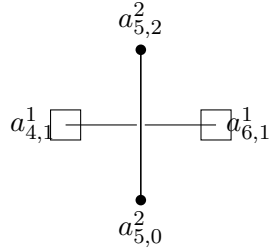
\begin{center}
\begin{tikzpicture}
\draw (-1,0)node(xline)[left] {$a^1_{4,1}$} -- (1,0 )node(xline)[right] {$a^1_{6,1}$};
\draw[draw=white,double=black,very thick] (0,-1)node(yline)[below] {$a^2_{5,0}$} -- (0,1)node(yline)[above] {$a^2_{5,2}$};
\fill (canvas cs:x=0cm,y=1cm) circle (2pt);
\fill (canvas cs:x=0cm,y=-1cm) circle (2pt);
\node at (-1,0) [transition]{};
\node at (1,0) [transition]{};
\end{tikzpicture}
\end{center}
\par For a given kernel solution $S$, we can draw all kernel diagrams for kernel equations containing its non-zero elements. This way we will have $Dgm(S)$ as a lattice diagram in which the non-zero lattice points of $S$ are connected to one another by line segments on account of their relations to other points through kernel equations. \par
\par
We say that two non-zero points are \emph{primarily kernel connected} if there exists a kernel equation containing both of them. A kernel diagram $Dgm(\varphi)$ is \emph{connected} if given any two points $f,g \in Dgm(\varphi)$, there exists points $(h_i)_{i=0}^{n} \in Dgm(\varphi)$ such that any two adjacent points in the sequence $f=h_0,h_1,...,h_n=g$ are primarily connected.  \par
We notice that for a given kernel element $\varphi = (\varphi^1, \varphi^2)$, at the point $(n,m)$ we have the element $(\varphi^1_{n,m}, \varphi^2_{n,m})$. Hence we see that in a connected component we have three possibility at a point $(n,m)$. These are 
\begin{enumerate}
\item $\varphi^1_{n,m}$
\item $\varphi^2_{n,m} $
\item 0.
\end{enumerate}
Hence we conclude that the kernel diagram of $\varphi$, $Dgm(\varphi)$ is a subset of  $\mathbb Z^2 \oplus \mathbb Z^2 \oplus \mathbb Z^2$. If we show that each of these connected component is a boundary then we are done proving our lemma. Hence, we shall concentrate on any connected component say, $\pi_i(Dgm(\varphi))$ and call it the diagram of the kernel element. This assumption causes no loss in the generality of the cases we need to consider. \par
\begin{lemma}
All kernel solutions are disjoint unions of closed graphs with no open edges as drawn below.
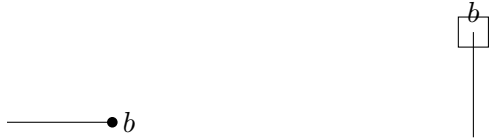
\begin{center}
\begin{tikzpicture}
\draw (-0.2,0) -- (1.2,0) node(xline)[right] {$b$};
\fill (canvas cs:x=1.2cm,y=0cm) circle (2pt);
\draw (6,-0.2) -- (6,1.2) node(yline)[above] {$b$};
\node at (6,1.2) [transition]{};
\end{tikzpicture}
\end{center} 
\end{lemma}
\begin{proof}

If such point \emph{b} was to exist in any of the kernel diagrams, which is \emph{kernel-connected} to a non-zero lattice point $b=w_{r,s}$ in the manner pictured above, then consider the kernel equation that contains this non-zero point $w_{r,s}$ and the other three zero points located at $(r+2,s),(r+1,s+1) \text{ and } (r+1,s-1)$. We derive a contradiction on the kernel condition.
\end{proof} \par
Now, the proof will proceed with induction over the \emph{number} of non-zero elements in a given kernel solution.\par
After going through a simple yet tedious process, one figures out that there are no kernel solution with the number of  non-zero entries less than or equal to 3. This can be seen as follows: \par
Consider three non-zero points $a^2_{n_1,m_1}, a^2_{n_2,m_2} ,a^1_{n_3,m_3}$ constituting a kernel solution. It is easy to check that $m_3 \neq m_1$. Let $m_3 < m_1$, then consider the following equation at $(n,m)=(n_1+1,m_1+2)$,
\begin{center}
$a^1_{n,m-1}-a^2_{n-1,m-2}={\lambda}^{m-1}a^1_{n-2,m-1}-{\lambda}^{1-n}a^2_{n-1,m}$.
\end{center}
We derive that for $a^2_{n_1,m_1}\neq 0$, it is imperative that $n_2=n_1$ and $m_2 = m_1+2$. Again consider the above equation at $(n,m)=(n_1+1,m_1+4)$. We find that $a^2_{n_2,m_2}=0$. This is a contradiction. \par

 For simplicity assume that $A_4$, a solution containing 4 non-zero points be centered at $(0,0)$. Now we study the four equations which contain these points. Hence we obtain the following relation over its non-zero points,
\begin{center}
$b=-\lambda a = -\lambda c = d$.
\end{center} The above relations mean that now $A_4$ can be written in terms of $Image(b_2)$ and the following diagram illustrates this.
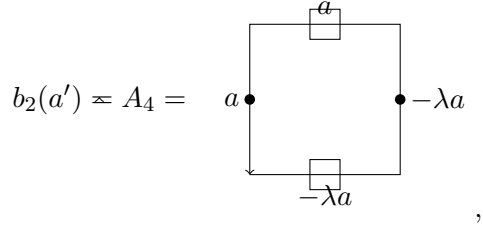
\begin{center}
\begin{tikzpicture}[->]
\draw (-1,-1) -- (1,-1)node[midway,below]{$-\lambda a$}-- (1,1)node[midway,right]{$-\lambda a$}--(-1,1)node[midway,above]{$a$}-- (-1,-1)node[midway,left]{$a$};
\node at (0,1) [transition]{};
\node at (0,-1) [transition]{};
\fill (canvas cs:x=1cm,y=0cm) circle (2pt);
\fill (canvas cs:x=-1cm,y=0cm) circle (2pt);
\draw (-3,0)node{$b_2(a')= A_4 =$};  
\end{tikzpicture},
\end{center}
where $a'_{0,0} =a$ and $a'_{r,s} =0$ for $(r,s) \neq (0,0)$. \par
Assume that all kernel solutions having the number of  non-zero elements less than or equal to $(x-1)$ come from image. Then consider a kernel solution $S_0$ with $x$ non-zero elements in it. Since this solution is finitely supported over the lattice plane, there exists a closed square region $\beta$  over which $S_0$ is supported.

Inside $\beta$ consider the left most column at least one point of which is non-zero. Choose the bottom point $\mu$ of this column. It is clear that $\mu = a^2_{r,s}$ for some $(r,s) \in \mathbb Z^2$. As if it were $a^1_{r,s}$ then consider the following kernel equation,
\begin{center}
$a^1_{r,s}-a^2_{r-1,s-1}={\lambda}^{s}a^1_{r-2,s}-{\lambda}^{1-r}a^2_{r-1,s+1}$ .
\end{center}
All but $a^1_{r,s}$ are zero. This is a contradiction.\par
We shall now construct a new solution $S_1$ from $S_0$, with the number of  non-zero elements in $S_1$ at most equal to $x$. Consider the following map $\wedge$
\begin{center}
$\wedge: \mathbb Z^2 \to \mathbb Z^2 , \text{ such that }$
\end{center}
\begin{center}
$(a^1_{r+1,{s-1}})^\wedge= 0, (a^2_{r,s})^\wedge= 0, (a^1_{{r+1},s+1})^\wedge=a^1_{{r+1},s+1} - a^2_{r,s}, \newline (a^2_{{r+2},{s}})^\wedge= a^2_{{r+2},{s}} +{\lambda}^{s+1} a^2_{r,{s}}  \text{ and } (a^j_{p,q})^\wedge= a^j_{p,q}$ for all other lattice points. 
\end{center}

\begin{lemma}Let $S_1:=\wedge(S_0)$. Then $S_1$ is a kernel solution such that $|S_1| \leq |S_0|$. If for some $u \in  {}_{-1}\mathcal A_\theta^{alg}$, $b_2(u)=S_1$ then $b_2(u')=S_0$, where $u' = u+ {}_{r,s}g$,
\begin{center}
${}_{r,s}g_{p,q}=\begin{cases}
a^2_{r,s} &\text{ if } (p,q)=(r+1,s)\\
0 & else. \end{cases}$
\end{center}
\end{lemma}
\begin{proof}
Second part is clear that the cardinality of non-zero entries is reduced by 2 when it is imposed that $(a^1_{r,s+2})^\wedge=(a^2_{r,s})^\wedge=0$ but then the cardinality may increase by two in case both $a^1_{r+1,s} \text{ and } a^2_{r+1,s}$ are zero in $S_0$.\par
To check that $S_1$ is a solution it is enough to check that all the kernel equations containing any of these four altered elements hold.
In other words we want to check that the following equation 
\begin{center}
$(a^1_{n,m})^\wedge-(a^2_{n-1,m-1})^\wedge-{\lambda}^{m}(a^1_{n-2,m})^\wedge+{\lambda}^{1-n}(a^2_{n-1,m+1})^\wedge=0$ 
\end{center}
holds for $(n,m)=(r+3,s+1), (r+1,s+1), (r+1,s-1), \text{ and } (r+3,s-1)$. \newline
\underline{Case 1: $(n,m)=(r+1,s+1)$} \newline
$(a^1_{r+1,s+1})^\wedge-(a^2_{r,s})^\wedge-{\lambda}^{s+1}(a^1_{r-1,s+1})^\wedge+{\lambda}^{-r}(a^2_{r,s+2})^\wedge=0 \newline
\implies a^1_{r+1,s+1}-a^2_{r,s}-0-{\lambda}^{s+1}a^1_{r-1,s+1}+{\lambda}^{-r}a^2_{r,s+2}=0$
This holds as in $S_0$ .\newline
\underline{Case 2 : $(n,m)=(r+1,s-1)$} \newline
$(a^1_{r+1,s-1})^\wedge-(a^2_{r,s-2})^\wedge-{\lambda}^{s-1}(a^1_{r-1,s-1})^\wedge+{\lambda}^{-r}(a^2_{r,s})^\wedge=0 \newline
\implies 0-0-0+0=0$ .\newline
\underline{Case 3 : $(n,m)=(r+3,s-1)$} \newline
$(a^1_{r+3,s-1})^\wedge-(a^2_{r+2,s-2})^\wedge-{\lambda}^{s-1}(a^1_{r+1,s-1})^\wedge+{\lambda}^{-2-r}(a^2_{r+2,s})^\wedge=0 \newline
\implies a^1_{r+3,s-1}-a^2_{r+2,s-2}-0+{\lambda}^{-2-r}(a^2_{r+2,s}+{\lambda}^{s+1}a^2_{r,s})=0 \newline
\implies a^1_{r+3,s-1}-a^2_{r+2,s-2}+{\lambda}^{-2-r}(a^2_{r+2,s}+{\lambda}^{s+1}a^2_{r,s})=0$. \newline
An easy exercise on $S_0$ reveals that $a^2_{r,s} = {\lambda}^{r} a^1_{r+1,s-1}$, hence we get the required relation. \newline
\underline{Case 4 : $(n,m)=(r+3,s+1)$} \newline
$(a^1_{r+3,s+1})^\wedge-(a^2_{r+2,s})^\wedge-{\lambda}^{s+1}(a^1_{r+1,s+1})^\wedge+{\lambda}^{-2-r}(a^2_{r+2,s+2})^\wedge=0 \newline
\implies a^1_{r+3,s+1}-a^2_{r+2,s}-{\lambda}^{s+1}a^2_{r,s}-{\lambda}^{s+1}(a^1_{r+1,s+1}-a^2_{r,s})+{\lambda}^{-2-r}a^2_{r+2,s+2}=0 \newline
\implies a^1_{r+3,s+1}-a^2_{r+2,s}-{\lambda}^{s+1}a^1_{r+1,s+1}+{\lambda}^{-2-r}a^2_{r+2,s+2}=0$. \newline
Above relation is indeed satisfied in $S_0$. \par
It is easy to see that like kernel diagram we can construct image diagrams. These diagrams are obtained by looking at how $b_2$-image of a non-zero lattice point look like as an element of $\mathcal A_\theta^{alg} \oplus \mathcal A_\theta^{alg}$. Given an element at $(0,0)$, its image are the four elements, two among $a^1$ and other two $a^2$. The smallest image diagram is the diagram of the kernel solution $A_4$, as described in previous pages. \par
Hence for each non-zero lattice element we have its image diagram similar to $A_4$, and hence for an element $a \in \mathcal A_\theta^{alg}$ we have its image diagram obtained by superimposing the image diagrams for each lattice points. For $H_1(\mathcal A_\theta^{alg}, {}_{-1}\mathcal A_\theta^{alg})$ to be zero it is a necessary condition that for each kernel element $\varphi=(\varphi^1, \varphi^2)$ there exists an image diagram similar to the construction of $Dgm(\varphi)$. Our induction process explicitly constructs the image candidate itself. \par
With this pictorial realization, it is clear from the diagram as well as from the explicit equations that  if we verify that
\begin{center}
${b_2(u')}_{p,q} = (a^1_{p,q},a^2_{p,q}) \text{ for } (p,q) = (r+1,s+1), (r,s), (r+1,s-1) \text{ and } (r+2,s)$
\end{center}
then we have proved that $b_2(u') = S_0$. \par
Observe that \newline
$(b_2(u'))_{r+1,s+1}=(b_2(u+{}_{rs}g))_{r+1,s+1}\newline=(b_2(u))_{r+1,s+1} + (b_2({}_{rs}g))_{r+1,s+1} = ((a^1_{r+1,s+1})^\wedge,(a^2_{r+1,s+1})^\wedge)+(a^2_{r,s},0) = (a^1_{r+1,s+1}, a^2_{r+1,s+1})$ ,\newline
$(b_2(u'))_{r,s}=(b_2(u+{}_{rs}g))_{r,s}\newline=(b_2(u))_{r,s} + (b_2({}_{rs}g))_{r,s} = ((a^1_{r,s})^\wedge,(a^2_{r,s})^\wedge)+(0,a^2_{r,s}) = (a^1_{r,s}, a^2_{r,s})$, \newline
$(b_2(u'))_{r+1,s-1}=(b_2(u+{}_{rs}g))_{r+1,s-1}\newline=(b_2(u))_{r+1,s-1} + (b_2({}_{rs}g))_{r+1,s-1} = ((a^1_{r+1,s-1})^\wedge,(a^2_{r+1,s-1})^\wedge)+(-{\lambda}^r a^2_{r,s},0) = (a^1_{r+1,s-1}, a^2_{r+1,s-1})$ \newline
and \newline
$(b_2(u'))_{r+2,s}=(b_2(u+{}_{rs}g))_{r+2,s}\newline=(b_2(u))_{r+2,s} + (b_2({}_{rs}g))_{r+2,s} = ((a^1_{r+2,s})^\wedge,(a^1_{r+2,s})^\wedge)+(0,-\lambda^{s+1}a^2_{r,s}) = (a^1_{r+2,s}, a^2_{r+2,s})$.
\end{proof}

\begin{lemma}
 $\wedge^N(S_0) = 0$ for any solution $S_0$ for a sufficiently large number $N$.
\end{lemma}
\begin{proof}
The lemma states that $|S_1| \leq |S_0|$ is strict after sufficiently many iterations. Notice that the above process reduces the number of non-zero entries in the left most column by 2. And when this process is iterated this will lead to shifting of the left most non-zero column rightwards. Pictorially if it is assumed that at no stage the iteration reduces the number of non-zero elements, then these many elements are compressed by the shifting of the left most non-zero column towards right by this process as the right edge of $\beta$ is not changed by this process unless all the non-zero elements are collected on this right edge of $\beta$. \par
Hence, all the non-zero elements of $\wedge^d(S_0)$ will at some stage be bounded to the right edge of $\beta$. But, then, consider the lowest element of $\wedge^d(S_0)$ on this edge. It has to be only $a^2_{w,t}$ for some $(w,t)$. Consider the following kernel equation. Every term in the equation
\begin{center}
$a^1_{w+1,t-1}-a^2_{w,t-2}-{\lambda}^{t-1}a^1_{w-1,t-1}+{\lambda}^{-w}a^2_{w,t}=0$ 
\end{center}
has all but $a^2_{w,t}$ equal zero, which is a contradiction.

Hence the induction process is complete.
\end{proof}



\begin{thm}
For $\mathbb Z_2$ action on $\mathcal A_\theta^{alg}$ we have the Hochschild homology groups as follows:
\begin{center}
$HH_0(\mathcal A_\theta^{alg} \rtimes \mathbb Z_2) \cong \mathbb C^5,
HH_1(\mathcal A_\theta^{alg} \rtimes \mathbb Z_2) \cong 0,
HH_2(\mathcal A_\theta^{alg} \rtimes \mathbb Z_2) \cong \mathbb C.$
\end{center}
\end{thm} 
\begin{proof}
We already know that $HH_0(\mathcal A_\theta^{alg})^{\mathbb Z_2} = \mathcal A_\theta^{alg}/ im(1 \otimes b_1)\cong \mathbb C$. We also know that $H_0(\mathcal A_\theta^{alg}, {}_{-1}\mathcal A_\theta^{alg})^{\mathbb Z_2} \cong \mathbb C^4$. Hence we have $HH_0(\mathcal A_\theta^{alg} \rtimes \mathbb Z_2) \cong \mathbb C^5$. \par
We also know from the previous section that $HH_2(\mathcal A_\theta^{alg})^{\mathbb Z_2} \cong \mathbb C$. As we know that $H_2(\mathcal A_\theta^{alg} , {}_{-1}\mathcal A_\theta) =0$, hence, we derive the desired dimension for $HH_2(\mathcal A_\theta^{alg} \rtimes \mathbb Z_2)$. \par
From the previous section we also know that $HH_1(\mathcal A_\theta^{alg} \rtimes \mathbb Z_2) \cong 0$. Combining this result with $H_1(\mathcal A_\theta^{alg}, {}_{-1}\mathcal A_\theta^{alg}) =0$, we get that $H_1(\mathcal A_\theta^{alg} \rtimes \mathbb Z_2) = 0$.
\end{proof}

\subsection{Cyclic homology of  $\mathcal A_\theta^{alg} \rtimes \mathbb Z_2$}
A. Connes introduced[C] an $S,B,I$ long exact sequence relating the Hochschild and cyclic homology of an algebra $A$,
\begin{center}
$... \xrightarrow{B} HH_n(A) \xrightarrow{I} HC_n(A) \xrightarrow{S} HC_{n-2}(A) \xrightarrow{B} HH_{n-1}(A) \xrightarrow{I}...$.\\
\end{center}
Since the $\mathbb Z_2$ action on $\mathcal A_{\theta, -1}^{alg}$ commutes with the map ${}_{-1}b$ , we obtain the following exact sequence
\begin{center}
$... \xrightarrow{B} (HH_n(\mathcal A^{alg}_{\theta,-1}))^{\mathbb Z_2} \xrightarrow{I} (HC_n(\mathcal A^{alg}_{\theta,-1}))^{\mathbb Z_2} \xrightarrow{S} (HC_{n-2}(\mathcal A^{alg}_{\theta, -1}))^{\mathbb Z_2} \xrightarrow{B} (HH_{n-1}(\mathcal A^{alg}_{\theta,-1}))^{\mathbb Z_2} \xrightarrow{I}...$.\\
\end{center}
\par
We know that $(HH_2(\mathcal A^{alg}_{\theta,-1}))^{\mathbb Z_2}=(HH_1(\mathcal A^{alg}_{\theta,-1}))^{\mathbb Z_2}=0$. Hence we obtain that
\begin{center}
$(HC_2(\mathcal A^{alg}_{\theta,-1}))^{\mathbb Z_2} \cong (HC_0(\mathcal A^{alg}_{\theta, -1}))^{\mathbb Z_2}$.
\end{center}
But, a preliminary result shows that $HH_0(\mathcal A_{\theta,-1}^{alg})=HC_0(\mathcal A_{\theta, -1}^{alg})$. Hence, we obtain that
\begin{center}
 $(HC_2(\mathcal A^{alg}_{\theta,-1}))^{\mathbb Z_2}=(HC_0(\mathcal A^{alg}_{\theta,-1}))^{\mathbb Z_2} \cong \mathbb C^4$.
\end{center}
Also, since $(HH_2(\mathcal A^{alg}_{\theta,-1}))^{\mathbb Z_2}=(HH_3(\mathcal A^{alg}_{\theta,-1}))^{\mathbb Z_2}=0$, we obtain that $(HC_3(\mathcal A^{alg}_{\theta,-1}))^{\mathbb Z_2}\cong (HC_1(\mathcal A^{alg}_{\theta,-1}))^{\mathbb Z_2}$. Since $HH_1(\mathcal A_{\theta,-1}^{alg})$ is trivial, we have $(HC_3(\mathcal A^{alg}_{\theta,-1}))^{\mathbb Z_2}\cong (HC_1(\mathcal A^{alg}_{\theta,-1}))^{\mathbb Z_2}=0$.
\end{proof}

\subsection{Periodic cyclic homology}
\begin{thm}
$HC_{even}(\mathcal A_\theta^{alg} \rtimes \mathbb Z_2) \cong \mathbb C^6 \text{ while } HC_{odd}(\mathcal A_\theta^{alg} \rtimes \mathbb Z_2) = 0$.
\end{thm}
\begin{proof}
It is clear from the above calculations that we have
\begin{center}
$(HC_\bullet(\mathcal A_{\theta,-1}^{alg}))^{\mathbb Z_2} \cong (HC_{\bullet-2}(\mathcal A_{\theta,-1}^{alg}))^{\mathbb Z_2}$
\end{center}
Hence we conclude that $HC_{even}(\mathcal A_\theta^{alg} \rtimes \mathbb Z_2) \cong \mathbb C^6$ \par
As for the odd cyclic homology, we have $(HC_3(\mathcal A^{alg}_{\theta,-1}))^{\mathbb Z_2}\cong (HC_1(\mathcal A^{alg}_{\theta,-1}))^{\mathbb Z_2}=0$, and we also have $HC_3(\mathcal A_\theta^{alg})^{\mathbb Z_2} = HC_1(\mathcal A_\theta^{alg})^{\mathbb Z_2}=HH_1(\mathcal A_\theta^{alg})^{\mathbb Z_2} = 0$. Combining these two results, we obtain $HC_{odd}(\mathcal A_\theta^{alg} \rtimes \mathbb Z_2) = 0$. So we have computed the Hochschild and cyclic homology of the noncommutative $\mathbb Z_2$ orbifold.
\end{proof}

\section{$\mathbb Z_3$ action on $\mathcal A_\theta^{alg}$}
The group $\mathbb Z_3$ is embedded in $SL(2,\mathbb Z)$ through its generator $g= \left[
 \begin{array}{cc}
   -1 & -1 \\
   1 & 0
 \end{array} \right]\in SL(2,\mathbb Z)$. The generator acts on $\mathcal A_\theta^{alg}$ in the following way
\begin{center}
$U_1 \mapsto U_2^{-1}, U_2 \mapsto \displaystyle \frac{U_1 U_2^{-1}}{\sqrt\lambda}$.
\end{center}
\subsection{Hochschild homology}..\newline
\begin{thm}
$HH_0((\mathcal {A}_{\theta, \omega^{\pm 1}}^{alg \bullet})^{\mathbb {Z}_3}) \cong \mathbb C^3$, while $HH_k((\mathcal {A}_{\theta, \omega^{\pm 1}}^{alg \bullet})^{\mathbb {Z}_3})$  are trivial groups for $k \geq 1$.
\end{thm}

\begin{lemma}
Consider the following chain complex $J_{\ast, \omega{\pm 1}}^{\mathbb Z_3}$
\begin{center}
$J_{\ast ,  \omega^{\pm 1}}^{\mathbb Z_3} := 0 \xleftarrow{{}_{ \omega^{\pm 1}}b} (\mathcal A^{alg}_{\theta,  \omega^{\pm 1}})^{\mathbb Z_3} \xleftarrow{{}_{ \omega^{\pm 1}}b} ((\mathcal A^{alg}_{\theta,  \omega^{\pm 1}})^{\otimes 2})^{\mathbb Z_3} \xleftarrow{{}_{ \omega^{\pm 1}}b} ((\mathcal A^{alg}_{\theta,  \omega^{\pm 1}})^{\otimes 3})^{\mathbb Z_3}...$\\
\end{center}
where,
\begin{center}
${}_{ \omega^{\pm 1}}b(a_0\otimes a_1\otimes....\otimes a_n)= b'(a_0\otimes a_1\otimes....\otimes a_n)+(-1)^n(( \omega^{\mp 1}\cdot a_n)a_0\otimes a_1\otimes....\otimes a_{n-1}).$\\
\end{center}
Then,
\begin{center}
$H_\bullet(J_{\ast , \omega^{\pm 1}}^{\mathbb Z_3}, {}_ {\omega^{\pm 1}}b) = (H_\bullet(J_\ast(\mathcal A_\theta^{alg}, {}_{ \omega^{\pm 1}}\mathcal A_\theta^{alg})), b)^{\mathbb Z_3}$
\end{center}
\end{lemma}

Hence using the adjusted Connes' complex for the algebraic case we can now calculate the homology groups.
\begin{center}
\begin{enumerate}
\item[$\bullet$]$H_0(\mathcal A_\theta^{alg},{}_{ \omega^{\pm 1}}\mathcal A^{alg}_\theta)= {}_{ \omega^{\pm 1}}\mathcal A^{alg}_\theta\otimes_{\mathfrak B_\theta^{alg}} \mathfrak B_\theta^{alg} / Image(1\otimes b_1),$
\item[$\bullet$]$H_1(\mathcal A_\theta^{alg},{}_{ \omega^{\pm 1}}\mathcal A^{alg}_\theta)= Ker(1\otimes b_1) / Image(1\otimes b_2),$
\item[$\bullet$]$H_2(\mathcal A_\theta^{alg},{}_{ \omega^{\pm 1}}\mathcal A^{alg}_\theta)= Ker(1\otimes b_2)$.
\end{enumerate}
\end{center}

\begin{lemma}$H_2(\mathcal A_\theta^{alg},{}_{ \omega^{\pm 1}}\mathcal A^{alg}_\theta) \cong 0.$
\end{lemma}
\begin{proof}
We prove here for the case $g= \omega$, the proof for $g=\omega^2$ is similar. Consider the map $(1 \otimes b_2)$ in the tensor complex. To calculate the kernel of this map we have a closer look at the differential.
\begin{center}
$(1\otimes b_2)(a\otimes I)=a\otimes_{\mathfrak B_\theta^{alg}}(U_2\otimes I-\lambda\otimes U_2)\otimes e_1-a\otimes_{\mathfrak B_\theta^{alg}}(\lambda U_1\otimes I-I\otimes U_1)\otimes e_2$.
\end{center}
Using the twisted bimodule structure of ${}_{ \omega}\mathcal A_\theta^{alg}$ over $\mathcal A_\theta^{alg}$, the equation can be simplified to the following,
\begin{center}
$(1\otimes b_2)(a\otimes I)= (aU_2-\sqrt{\lambda}U_1U_2^{-1}a, U_2^{-1}a-\lambda a U_1 )$.
\end{center}
Hence we obtain the following relation over an element $(a \otimes 1)$ to reside in $ker( 1\otimes b_2)$. 
\begin{center}
$H_2(\mathcal A_\theta^{alg},{}_{\omega}\mathcal A^{alg}_\theta)= \left\lbrace a\in {}_{\omega}\mathcal A_\theta^{alg} | a_{n,{m-1}}={\lambda}^{1.5-m}a_{n-1,m+1}; {\lambda}^{m+1}a_{n-1,m}=\lambda^n a_{n,m+1}\right\rbrace$.
\end{center}
Since, no such nontrivial element exists in ${}_{\omega}\mathcal A_\theta^{alg}$. If such a non-trivial element existed, then all $a_{m,n}$ have to be same up to multiple of $\lambda$. But since the algebra consists of finitely supported elements, they are reduced to zero. So, we have the desired result.
\begin{center}
$H_2(\mathcal A_\theta^{alg},{}_{\omega}\mathcal A^{alg}_\theta)=0$.
\end{center}
\end{proof}
\begin{lemma} $H_0(\mathcal A_\theta^{alg},{}_{\omega}\mathcal A^{alg}_\theta)^{\mathbb Z_2} \cong \mathbb C^3.$
\end{lemma}
\begin{proof}
We have the map $(1 \otimes b_1)$ in the tensor complex defined below,
\begin{center}
$(1\otimes b_1)(a\otimes I\otimes e_j)=a\otimes_{\mathfrak B_\theta^{alg}}(I\otimes U_j-U_j\otimes I)= aU_j-U_j^{-1}a$.
\end{center}
As before if we use the twisted bicomplex structure of ${}_{\omega}\mathcal A_\theta^{alg}$ over the algebra $\mathcal A_\theta^{alg}$, the map$(1 \otimes b_1)$ can be simplified as follows,
\begin{center}
$b_1(a_1,0)= a_1U_1- U_2^{-1}a_1\text{ and } b_1(0,a_2)= a_2 U_2-\displaystyle \frac{U_1U_2^{-1}}{\sqrt{\lambda}}a_2$.
\end{center}
Observe that  $b_1(a_1,a_2)= b_1(a_1,0)+b_1(0,a_2)$. Furthermore we can simplify the calculation by considering only elements of the type $b_1(U_1^nU_2^m,0)$, and similarly for $b_1(0,a_2)$ we can consider the elements of the type $b_1(0,U_1^nU_2^m)$.
We observe that 
\begin{center}
$b_1(U_1^nU_2^m,0)= ({U_1}^n{U_2}^m)U_1-U_2^{-1}({U_1}^n{U_2}^m)={\lambda}^{-n}({\lambda}^{m}U_1U_2-1)U_1^{n}U_2^{m-1}$
\end{center}
and
\begin{center}
$b_1(0,U_1^nU_2^{m-1})= ({U_1}^n{U_2}^{m-1})U_2-\frac{U_1 U_2^{-1}}{\sqrt\lambda}({U_1}^n{U_2}^{m-1})= ({\lambda}^{-2n}U_2^2-U_1)U_1^nU_2^{m-2}.$
\end{center}
From the above relations we see that in this quotient space $\bar{a_{1,0}} =\bar{ a_{0,-1}}=\bar{ a_{0,2}}$. Observe that $\langle \bar{a_{0,0}}, \bar{a_{1,0}}, \bar{a_{0,1}} \rangle$ is a basis of this quotient space. \par
Now to obtain $(H_0(\mathcal A_\theta^{alg}, {}_{\omega}\mathcal A_\theta^{alg}))^{\mathbb Z_3}$ we consider complex map $h: J_\ast \to C_\ast$, since the map $h_0 : J_0(\mathcal A_\theta^{alg}) \to C_0(\mathcal A_\theta^{alg})$ is the identity map, the $\mathbb Z_3$ action on the bar complex is translated to the Kozul complex with no alteration. Hence we get that

\begin{center}
$H_0(\mathcal A_\theta^{alg},{}_{\omega}\mathcal A^{alg}_\theta)^{\mathbb {Z}_3}= \langle \bar{a_{0,0}}, \bar{a_{1,0}}, \bar{a_{0,1}} \rangle.$\\
\end{center}
\end{proof}

\begin{lemma} $H_1(\mathcal A_\theta^{alg},{}_{\omega}\mathcal A^{alg}_\theta) \cong 0$.
\end{lemma}
\begin{proof}
 From the previous calculations we do have an explicit formula for the kernel and the image equations of $H_1(\mathcal A_\theta^{alg}, {}_{\omega}\mathcal A_\theta^{alg})$.
\begin{center}
$\lambda^{m-1}a^1_{n-1,m-1}-\lambda^{-n}a^1_{n,m}={\lambda}^{0.5-n}a^2_{n-1,m}-a^2_{n,m-2}$ ( Kernel condition ),
\end{center}
\begin{center}
$a^1_{n,m-1}={\lambda}^{1.5-n}a_{{n-1},m}-a_{n.m-2};\text{ } a^2_{n-1,m-1}={\lambda}^{m}a_{n-2,m-1}-{\lambda}^{1-n}a_{{n-1},m}$( Image of $a\in \mathcal A^{alg}_\theta)$.
\end{center}
As before, \emph{a,b,c, and d} are said to be \emph {kernel-connected} if there exists a kernel equation containing them. So, any given kernel solution $A$ we have a diagram associated to it, extended by plotting  all the kernel equations that contain any of its non zero points, which for simplicity we denote by $A$. \newline
e.g. equation $\lambda^{0}a^1_{-1,0}-\lambda^{-0}a^1_{0,1}={\sqrt\lambda}a^2_{-1,1}-a^2_{0,-1}$ is drawn below: 
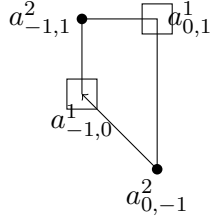
\begin{center}
\begin{tikzpicture}[->]
\draw (-1,0) -- (-1,1)node[left]{$a^2_{-1,1}$} -- (0,1)node[above, right]{$a^1_{0,1}$}-- (0,-1)node[below]{$a^2_{0,-1}$}--(-1,0)node[below]{$a^1_{-1,0}$};
\node at (0,1) [transition]{};
\node at (-1,0) [transition]{};
\fill (-1,1) circle (2pt);
\fill (0,-1) circle (2pt);
\end{tikzpicture}.
\end{center}

Now, the proof will proceed with induction over the \emph{number} of non-zero elements in a given kernel solution.\par
After going through a simple yet tedious process, one figures out that there are no kernel solution with the number of  non-zero entries less than or equal to 3. This can be seen as follows: \par
Consider three non-zero points $a^2_{n_1,m_1}, a^2_{n_2,m_2} ,a^1_{n_3,m_3}$ constituting a kernel solution. It is easy to check that $m_2 \neq m_1$. Let $m_1 < m_2$, then consider the following equation at $(p,q)=(n_1+1,m_1)$,
\begin{center}
$\lambda^{q-1}a^1_{p-1,q-1}-\lambda^{-p}a^1_{p,q}={\lambda}^{0.5-p}a^2_{p-1,q}-a^2_{p,q-2}$.
\end{center}
 then we see that $m_3=m_1$. Considering the above equation at $(p,q) = (n_2,m_2)$ we get, $n_3=n_2$ and $m_3= m_2-1$. Now finally we consider the above kernel equation at $(p,q)=(n_1-1, m_1+2)$, on which we get that $a^2_{n_1,m_1} = 0$, which is a contradiction.\par
Assume that all kernel solutions having number of  non-zero elements less than or equal to $(x-1)$ come from image. Then consider a kernel solution $S_0$ with $x$ non-zero elements in it. Since this solution is finitely supported over the lattice plane, there exists a closed square region $\beta$  over which $S_0$ is supported.

Inside $\beta$ consider the left most column at least one point of which is not zero. Choose the bottom point $\mu$ of this column. It is clear that $\mu = a^2_{r,s}$ for some $(r,s) \in \mathbb Z^2$. As if it were $a^1_{r,s}$ then consider the following kernel equation.
\begin{center}
$\lambda^{s-1}a^1_{r-1,s-1}-\lambda^{-r}a^1_{r,s}={\lambda}^{0.5-r}a^2_{r-1,s}-a^2_{r,s-2}$.
\end{center}
All but $a^1_{r,s}$ are zero. This is a contradiction.\par
We shall now construct a new solution $S_1$ from $S_0$, with the number of  non-zero elements in $S_1$ at most equal to $x$. Consider the following map $\wedge$
\begin{center}
$\wedge: \mathbb Z^2 \to \mathbb Z^2 , \text{ such that }$
\end{center}
\begin{center}
$(a^1_{r,{s+2}})^\wedge= 0, (a^2_{r,s})^\wedge= 0, (a^1_{{r+1},s})^\wedge=a^1_{{r+1},s} + {\sqrt\lambda}a^2_{r,s}, \newline (a^2_{{r+1},{s+1}})^\wedge= a^2_{{r+1},{s+1}} +{\lambda}^{s+2} a^1_{r,{s+2}}  \text{ and } (a^j_{p,q})^\wedge= a^j_{p,q}$ for all other lattice points. 
\end{center}

\begin{lemma}Let $S_1:=\wedge(S_0)$. Then $S_1$ is a kernel solution such that $|S_1| \leq |S_0|$. If for some $u \in  {}_{\omega}\mathcal A_\theta^{alg}$, $b_2(u)=S_1$ then $b_2(u')=S_0$, where $u' = u+ {}_{r,s}g$,
\begin{center}
${}_{r,s}g_{p,q}=\begin{cases}
-a^1_{r,s+2} &\text{ if } (p,q)=(r,s+1)\\
0 & else. \end{cases}$
\end{center}
\end{lemma}
\begin{proof}
The second part is clear that the cardinality of non-zero entries is reduced by 2 when it is asked that $(a^1_{r,s+2})^\wedge=(a^2_{r,s})^\wedge=0$ but then the cardinality may increase by two in case both $a^1_{r+1,s} \text{ and } a^2_{r+1,s}$ are zero in $S_0$.\par
To check that $S_1$ is a solution it is enough to check that all the kernel equations containing any of these four altered elements hold.
In other words we want to check that the following equation 
\begin{center}
$\lambda^{m}(a^1_{n-1,m})^\wedge-\lambda^{-n}(a^1_{n,m+1})^\wedge={\lambda}^{0.5-n}(a^2_{n-1,m+1})^\wedge-(a^2_{n,m-1})^\wedge$
\end{center}
holds for $(n,m)=(r,s+1), (r+1,s+2), (r+1,s-1) \text{ and } (r+2,s+1)$. \newline
\underline{Case 1: $(n,m)=(r,s+1)$} \newline
$\lambda^{{s+1}}(a^1_{r-1,{s+1}})^\wedge-\lambda^{-r}(a^1_{r,s+2})^\wedge={\lambda}^{0.5-r}(a^2_{r-1,s+2})^\wedge-(a^2_{r,s})^\wedge \newline
\implies \lambda^{{s+1}}(a^1_{r-1,{s+1}})^\wedge = {\lambda}^{0.5-r}(a^2_{r-1,s+2})^\wedge$,
which holds as : $a^1_{r-1,s+1}=a^2_{r-1,s+1}=0$.\newline
\underline{Case 2 : $(n,m)=(r+1,s+2)$} \newline
$\lambda^{{s+2}}(a^1_{r,{s+2}})^\wedge-\lambda^{-1-r}(a^1_{{r+1},s+3})^\wedge-{\lambda}^{-0.5-r}(a^2_{r,s+3})^\wedge+(a^2_{{r+1},{s+1}})^\wedge \newline
 = 0-\lambda^{-1-r}a^1_{r+1,s+3}-{\lambda}^{-0.5-r}a^2_{r,s+3}+ a^2_{{r+1},{s+1}} + {\lambda}^{s+2} a^1_{r,{s+2}}$.\newline
The above equation holds in $S_0$.
\newline
\underline{Case 3 : $(n,m)=(r+2,s)$}\newline
$\lambda^{s}(a^1_{r+1,s})^\wedge-\lambda^{-r-2}(a^1_{r+2,s+1})^\wedge-{\lambda}^{-1.5-r}(a^2_{r+1,s+1})^\wedge+(a^2_{r+2,s-1})^\wedge =0\newline
\implies \lambda^{s}a^1_{r+1,s} +{\lambda}^{s+0.5} a^2_{r,s} -\lambda^{-r-1}a^2_{r+2,s+1}-\lambda^{-0.5-r} a^2_{{r+1},{s+1}} -{\lambda}^{s-r+0.5} a^1_{r,{s+2}} - a^2_{r+2,s-1} =0\newline
\implies {\lambda}^{s+1.5} a^2_{r,s} = {\lambda}^{s-r+1.5} a^1_{r,{s+2}}$ \newline
The above is true if one considers the kernel equation.
\begin{center}
${\lambda}^{s+1}a^1_{r-1,s+1}-{\lambda}^{-r}a^1_{r,s+2}+a^2_{r,s} - \lambda^{0.5-r}a^2_{r-1,s+2}=0$
\end{center}
$a^2_{r-1,s+2}=a^2_{r-1,s+1}=0 \newline
\implies a^2_{r,s} = {\lambda}^{-r}a^1_{r,s+2}$ , which is what we required.
\newline
\underline{Case 4 : $(n,m)=(r+1,s-1)$} \newline
$\lambda^{s-1}(a^1_{r,{s-1}})^\wedge-\lambda^{-1-r}(a^1_{{r+1},s})^\wedge-{\lambda}^{-0.5-r}(a^2_{r,s})^\wedge+(a^2_{{r+1},{s-2}})^\wedge \newline
 =\lambda^{s-1}a^1_{r,{s-1}}-\lambda^{-1-r}a^1_{{r+1},s}-{\lambda}^{-1-r}(\sqrt\lambda a^2_{r,s})-0+a^2_{{r+1},{s-2}}\newline
=\lambda^{s-1}a^1_{r,{s-1}}-\lambda^{-1-r}a^1_{{r+1},s}-{\lambda}^{-0.5-r}a^2_{r,s}+a^2_{{r+1},{s-2}}$.\newline
The above is a kernel equation in $S_0$, and we have shown that $S_1:= \wedge(S_0)$ is a kernel solution. \par
It is easy to see that like kernel equation diagram there is also an image solution diagram indicating how image elements are formed given an element at $(0,0)$,
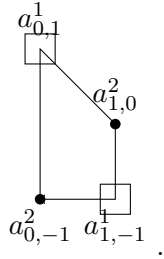
\begin{center}
\begin{tikzpicture}[->]
\draw (0,-1)-- (0,1)node[right,above,above]{$a^1_{0,1}$} -- (1,0)node[above]{$a^2_{1,0}$}-- (1,-1)node[below]{$a^1_{1,-1}$}--(0,-1)node[below]{$a^2_{0,-1}$};
\node at (1,-1) [transition]{};
\node at (0,1) [transition]{};
\fill (1,0) circle (2pt);
\fill (0,-1) circle (2pt); 
\end{tikzpicture}.
\end{center}
So, it is clear from the diagram as well as from the explicit equations that we have for the map $b_2$, an image element $a_{0,0}$ induces kernel solution elements to its right-below and up for $a^1$, or right and below for $a^2$ elements. Hence if we verify that
\begin{center}
${b_2(u')}_{p,q} = (a^1_{p,q},a^2_{p,q}) \text{ for } (p,q) = (r,s+2), (r,s), (r+1,s+1), (r+1,s)$
\end{center}
then we have proved that $b_2(u') = S_0$. \par
Observe that \newline
$(b_2(u'))_{r,s+2}=(b_2(u+{}_{rs}g))_{r,s+2}\newline=(b_2(u))_{r,s+2} + (b_2({}_{rs}g))_{r,s+2} = ((a^1_{r,s+2})^\wedge,(a^2_{r,s+2})^\wedge)+(a^1_{r,s+2},0) = (a^1_{r,s+2}, a^2_{r,s+2})$ ,\newline
$(b_2(u'))_{r,s}=(b_2(u+{}_{rs}g))_{r,s}\newline=(b_2(u))_{r,s} + (b_2({}_{rs}g))_{r,s} = ((a^1_{r,s})^\wedge,(a^2_{r,s})^\wedge)+(0,\lambda^r a^1_{r,s+2}) = (a^1_{r,s}, a^2_{r,s})$, \newline
$(b_2(u'))_{r+1,s+1}=(b_2(u+{}_{rs}g))_{r+1,s+1}\newline=(b_2(u))_{r+1,s+1} + (b_2({}_{rs}g))_{r+1,s+1} = ((a^1_{r+1,s+1})^\wedge,(a^2_{r+1,s+1})^\wedge)+(0,-\lambda^{s+2}a^1_{r,s+2}) = (a^1_{r+1,s+1}, a^2_{r+1,s+1})$ \newline
and \newline
$(b_2(u'))_{r+1,s}=(b_2(u+{}_{rs}g))_{r+1,s}\newline=(b_2(u))_{r+1,s} + (b_2({}_{rs}g))_{r+1,s} = ((a^1_{r+1,s})^\wedge,(a^1_{r+1,s})^\wedge)+(\lambda^{0.5-r}a^1_{r,s+2},0) = (a^1_{r+1,s}, a^2_{r+1,s})$.
\end{proof}

\begin{lemma}
 $\wedge^N(S_0) = 0$ for any solution $S_0$ for a sufficiently large number $N$.
\end{lemma}
\begin{proof}
The lemma states that $|S_1| \leq |S_0|$ is strict after sufficiently many iterations. Notice that the above process reduces the number of non-zero entries in the left most column by 2. And when this process is iterated this will lead to shifting of the left most non-zero column rightwards. Pictorially if it is assumed that at no stage the iteration reduces the number of non-zero elements, then these many elements are compressed by the shifting of the left most non-zero column towards right. As the right edge of $\beta$ is not changed by this process unless all the non-zero elements are collected on this right edge of $\beta$, all the non-zero elements of $\wedge^d(S_0)$ will at some stage be bounded to the right edge of $\beta$. But, then, consider the lowest element of $\wedge^d(S_0)$ on this edge. It has to be only $a^2_{w,t}$ for some $(w,t)$. Consider the following kernel equation,
\begin{center}
$\lambda^{1+t}a^1_{w-1,t+1}-\lambda^{-w}a^1_{w,t+2}=-a^2_{w,t}+\lambda^{0.5-w}a^2_{w-1,t+1}$ .
\end{center}
This equation has all but $a^2_{w,t}$ as zero, which is a contradiction.

Hence the induction process is complete.
\end{proof}
\end{proof}

\begin{thm}
For $\mathbb Z_3$ action on $\mathcal A_\theta^{alg}$ we have the Hochschild homology groups as follows:
\begin{center}
$HH_0(\mathcal A_\theta^{alg} \rtimes \mathbb Z_3) \cong \mathbb C^7,
HH_1(\mathcal A_\theta^{alg} \rtimes \mathbb Z_3) \cong 0,
HH_2(\mathcal A_\theta^{alg} \rtimes \mathbb Z_3) \cong \mathbb C.$
\end{center}
\end{thm} 
\begin{proof}
From the calculations in the previous sections regarding the $\mathbb Z_3$ invariance we know that $HH_0(\mathcal A_\theta^{alg})^{\mathbb Z_3} \cong \mathbb C$ and similarly one can check that $H_0(\mathcal A_\theta^{alg}, {}_{\pm\omega}\mathcal A_\theta^{alg})^{\mathbb Z_3} \cong \mathbb C^3$. Hence we have the following result
\begin{center}
$HH_0(\mathcal A_\theta^{alg} \rtimes \mathbb Z_3) \cong \mathbb C^7$.
\end{center}
\par  As $HH_2(\mathcal A_\theta^{alg})^{\mathbb Z_3} \cong \mathbb C$. And we have $H_2(\mathcal A_\theta^{alg} , {}_{\pm\omega}\mathcal A_\theta^{alg}) \cong 0$, we conclude that 
\begin{center}
$HH_2(\mathcal A_\theta^{alg} \rtimes \mathbb Z_2) \cong \mathbb C$.
\end{center}
Since $HH_1(\mathcal A_\theta^{alg})^{\mathbb Z_3}=H_1(\mathcal A_\theta^{alg}, {}_{\omega^{\pm 1}}\mathcal A_\theta^{alg}) =0$, we obtain that $HH_1(\mathcal A_\theta \rtimes \mathbb Z_3) =0$.
\end{proof}

\subsection{Cyclic homology of  $\mathcal A_\theta^{alg} \rtimes \mathbb Z_3$}.\newline
\begin{thm}
$HC_{even}(\mathcal A_{\theta, \omega^{\pm 1}}^{alg}) \cong \mathbb C^3$, while 
$HC_{odd}(\mathcal A_{\theta, \omega^{\pm 1}}^{alg}) = 0.$
\end{thm}
\begin{proof}
We apply the $S,B,I$ long exact sequence relating the Hochschild and cyclic homology of an algebra $A$,
\begin{center}
$... \xrightarrow{B} HH_n(A) \xrightarrow{I} HC_n(A) \xrightarrow{S} HC_{n-2}(A) \xrightarrow{B} HH_{n-1}(A) \xrightarrow{I}...$.\\
\end{center}
Since the $\mathbb Z_3$ action on $\mathcal A_{\theta, \omega^{\pm 1}}^{alg}$ commutes with the map ${}_{\omega^{\pm 1}}b$ , we obtain the following exact sequence
\begin{center}
$... \xrightarrow{B} (HH_n(\mathcal A^{alg}_{\theta,\omega^{\pm 1}}))^{\mathbb Z_3} \xrightarrow{I} (HC_n(\mathcal A^{alg}_{\theta,\omega^{\pm 1}}))^{\mathbb Z_3} \xrightarrow{S} (HC_{n-2}(\mathcal A^{alg}_{\theta, \omega^{\pm 1}}))^{\mathbb Z_3} \xrightarrow{B} (HH_{n-1}(\mathcal A^{alg}_{\theta,\omega^{\pm 1}}))^{\mathbb Z_3} \xrightarrow{I}...$.\\
\end{center}
\par
We know that $(HH_2(\mathcal A^{alg}_{\theta,\omega^{\pm 1}}))=(HH_1(\mathcal A^{alg}_{\theta,\omega^{\pm 1}}))=0$. Hence we obtain that
\begin{center}
$(HC_2(\mathcal A^{alg}_{\theta,\omega^{\pm 1}}))^{\mathbb Z_3} \cong (HC_0(\mathcal A^{alg}_{\theta, \omega^{\pm 1}}))^{\mathbb Z_3}$.
\end{center}
But, a preliminary result shows that $HH_0(\mathcal A_{\theta,\omega^{\pm 1}}^{alg})=HC_0(\mathcal A_{\theta, \omega^{\pm 1}}^{alg})$. Hence, we obtain that
\begin{center}
 $(HC_2(\mathcal A^{alg}_{\theta,\omega^{\pm 1}}))^{\mathbb Z_3}=(HC_0(\mathcal A^{alg}_{\theta,\omega^{\pm 1}}))^{\mathbb Z_3} \cong \mathbb C^3$.
\end{center}
Also, since $(HH_2(\mathcal A^{alg}_{\theta,\omega^{\pm 1}}))^{\mathbb Z_3}=(HH_3(\mathcal A^{alg}_{\theta,\omega^{\pm 1}}))^{\mathbb Z_3}=0$, we obtain that 
\begin{center}$(HC_3(\mathcal A^{alg}_{\theta,\omega^{\pm 1}}))^{\mathbb Z_3}\cong (HC_1(\mathcal A^{alg}_{\theta,\omega^{\pm 1}}))^{\mathbb Z_3}$.
\end{center} \par
Since $HH_1(\mathcal A_{\theta,\omega^{\pm 1}}^{alg})$ is trivial, we have 
\begin{center}
$(HC_3(\mathcal A^{alg}_{\theta,\omega^{\pm 1}}))^{\mathbb Z_3}\cong (HC_1(\mathcal A^{alg}_{\theta,\omega^{\pm 1}}))^{\mathbb Z_3}=0$.
\end{center}
\end{proof}

\subsection{Periodic cyclic homology}.\newline
\begin{thm}
$HC_{even}(\mathcal A_\theta^{alg} \rtimes \mathbb Z_3) \cong \mathbb C^8 \text{ while } HC_{odd}(\mathcal A_\theta^{alg} \rtimes \mathbb Z_3) = 0$.
\end{thm}
\begin{proof}
We have
\begin{center}
$... \xrightarrow{B} (HH_2(\mathcal A^{alg}_{\theta}))^{\mathbb Z_3} \xrightarrow{I} (HC_2(\mathcal A^{alg}_{\theta}))^{\mathbb Z_3} \xrightarrow{S} (HC_{0}(\mathcal A^{alg}_{\theta}))^{\mathbb Z_3} \xrightarrow{B} (HH_{1}(\mathcal A^{alg}_{\theta,}))^{\mathbb Z_3} \xrightarrow{I}...$.
\end{center}
Since $HH_2(\mathcal A_\theta^{alg})^{\mathbb Z_3} \cong HC_0(\mathcal A_\theta^{alg})^{\mathbb Z_3} \cong \mathbb C$. Also we have the following isomorphism of the homology groups $(HC_{even}(\mathcal A_{\theta, \omega^{\pm 1}}^{alg}))^{\mathbb Z_3} \cong (HC_{2}(\mathcal A_{\theta, \omega^{\pm 1}}^{alg}))^{\mathbb Z_3} \cong \mathbb C^3$. Hence we conclude that 
\begin{center}
$HC_{even}(\mathcal A_\theta^{alg} \rtimes \mathbb Z_3) \cong \mathbb C^{8}$.
\end{center}
As for the odd cyclic homology, we have $(HC_3(\mathcal A^{alg}_{\theta,\omega^{\pm 1}}))^{\mathbb Z_3}\cong (HC_1(\mathcal A^{alg}_{\theta,\omega^{\pm 1}}))^{\mathbb Z_3}=0$, and we also have $HC_3(\mathcal A_\theta^{alg})^{\mathbb Z_3} = HC_1(\mathcal A_\theta^{alg})^{\mathbb Z_3}=HH_1(\mathcal A_\theta^{alg})^{\mathbb Z_3} = 0$. Combining these two results, we obtain that 
\begin{center}
$HC_{odd}(\mathcal A_\theta^{alg} \rtimes \mathbb Z_3) = 0$.
\end{center}
\par So we have computed the Hochschild and cyclic homology of the $\mathbb Z_3$ orbifold.
\end{proof}

\section{$\mathbb Z_4$ action on $\mathcal A_\theta^{alg}$}
The group $\mathbb Z_4$ is embedded in $SL(2,\mathbb Z)$ through its generator $g= \left[
 \begin{array}{cc}
   0 & -1 \\
   -1& 0
 \end{array} \right]\in SL(2,\mathbb Z)$. The generator acts of $\mathcal A_\theta^{alg}$ in the following way
\begin{center}
$U_1 \mapsto U_2^{-1}, U_2 \mapsto U_1$.
\end{center}
\subsection{Hochschild homology}.\newline
\begin{thm}
Let $i$ be the generator of $\mathbb Z_4$, we have $HH_0((\mathcal {A}_{\theta, {\pm i}}^{alg \bullet})^{\mathbb {Z}_4}) \cong \mathbb C^2$, while \newline
$HH_k((\mathcal {A}_{\theta, {\pm i}}^{alg \bullet})^{\mathbb {Z}_4})$  are trivial groups for $k \geq 1$.
\end{thm}

\begin{lemma}
Consider the following chain complex $J_{\ast, \pm i}^{\mathbb Z_4}$
\begin{center}
$J_{\ast ,  \pm i}^{\mathbb Z_4} := 0 \xleftarrow{{}_{ \pm i}b} (\mathcal A^{alg}_{\theta, \pm i})^{\mathbb Z_4} \xleftarrow{{}_{ \pm i}b} ((\mathcal A^{alg}_{\theta,  \pm i})^{\otimes 2})^{\mathbb Z_4} \xleftarrow{{}_{ \pm i}b} ((\mathcal A^{alg}_{\theta,  \pm i})^{\otimes 3})^{\mathbb Z_4}...$
\end{center}
where,
\begin{center}
${}_{ \pm i}b(a_0\otimes a_1\otimes....\otimes a_n)= b'(a_0\otimes a_1\otimes....\otimes a_n)+(-1)^n(( \mp i\cdot a_n)a_0\otimes a_1\otimes....\otimes a_{n-1})$.
\end{center}
Then,
\begin{center}
$H_\bullet(J_{\ast , \pm i}^{\mathbb Z_4}, {}_ {\pm i}b) = (H_\bullet(J_\ast(\mathcal A_\theta^{alg}, {}_{ \pm i}\mathcal A_\theta^{alg})), b)^{\mathbb Z_4}$.
\end{center}
\end{lemma}

Hence using the adjusted Connes' complex for the algebraic case we can now calculate the homology groups.
\begin{center}
\begin{enumerate}
\item[$\bullet$]$H_0(\mathcal A_\theta^{alg},{}_{ \pm i}\mathcal A^{alg}_\theta)= {}_{\pm i}\mathcal A^{alg}_\theta\otimes_{\mathfrak B_\theta^{alg}} \mathfrak B_\theta^{alg} / Image(1\otimes b_1),$
\item[$\bullet$]$H_1(\mathcal A_\theta^{alg},{}_{ \pm i}\mathcal A^{alg}_\theta)= Ker(1\otimes b_1) / Image(1\otimes b_2),$
\item[$\bullet$]$H_2(\mathcal A_\theta^{alg},{}_{\pm i}\mathcal A^{alg}_\theta)= Ker(1\otimes b_2)$.
\end{enumerate}
\end{center}

\begin{lemma} $H_2(\mathcal A_\theta^{alg},{}_{i}\mathcal A^{alg}_\theta) \cong 0.$
\end{lemma}
\begin{proof}
We prove here for case $g= i$, the proof for $g=-i$ is similar. Consider the map $(1 \otimes b_2)$ in the tensor complex. To calculate the kernel of this map we have a closer look at this map.
\begin{center}
$(1\otimes b_2)(a\otimes I)=a\otimes_{\mathfrak B_\theta^{alg}}(U_2\otimes I-\lambda\otimes U_2)\otimes e_1-a\otimes_{\mathfrak B_\theta^{alg}}(\lambda U_1\otimes I-I\otimes U_1)\otimes e_2$
\end{center}
Using the twisted bimodule structure of ${}_{ i}\mathcal A_\theta^{alg}$ over $\mathcal A_\theta^{alg}$, we simplify the above equation to the following,
\begin{center}
$(1\otimes b_2)(a\otimes I)= (a U_2-\lambda U_1a , U_2^{-1}a-\lambda aU_1)$.
\end{center}
Hence we obtain the following relation over an element $(a \otimes 1)$ to reside in $ker( 1\otimes b_2)$. 
\begin{center}
$H_2(\mathcal A_\theta^{alg},{}_{\omega}\mathcal A^{alg}_\theta)= \left\lbrace a\in {}_{i}\mathcal A_\theta^{alg} | \lambda a_{n-1,m}=a_{n,m-1}; {\lambda}^{m-n}a_{n-1,m}=\lambda^2 a_{n,m+1}\right\rbrace$ .
\end{center}
Since, no such nontrivial element exists in ${}_{i}\mathcal A_\theta^{alg}$ because if it did then $|a _{n,m}|=|a_{n,m+2}|$. So, we have the desired result,
\begin{center}
$H_2(\mathcal A_\theta^{alg},{}_{i}\mathcal A^{alg}_\theta)=0$.
\end{center}
\end{proof}
\begin{lemma} $H_0(\mathcal A_\theta^{alg},{}_{i}\mathcal A^{alg}_\theta)^{\mathbb Z_4} \cong \mathbb C^2.$
\end{lemma}
\begin{proof}
We have the map $(1 \otimes b_1)$ in the tensor complex defined below,
\begin{center}
$(1\otimes b_1)(a\otimes I\otimes e_j)=a\otimes_{\mathfrak B_\theta^{alg}}(I\otimes U_j-U_j\otimes I)= U_j^{-1}a-aU_j$.
\end{center}
As before if we use the twisted bicomplex structure of ${}_{i}\mathcal A_\theta^{alg}$ over the algebra $\mathcal A_\theta^{alg}$, the map$(1 \otimes b_1)$ can be simplified as follows,
\begin{center}
$b_1(a_1,0)= a_1U_1- U_2^{-1}a_1\text{ and } b_1(0,a_2)= a_2 U_2-U_1a_2$.
\end{center}
Observe that  $b_1(a_1,a_2)= b_1(a_1,0)+b_1(0,a_2)$. Further we can simplify the calculation by considering only elements of the type $b_1(U_1^nU_2^m,0)$, and similarly for $b_1(0,a_2)$ we can consider the elements of the type $b_1(0,U_1^nU_2^m)$.
We observe that 
\begin{center}
$b_1(U_1^nU_2^m,0)= ({U_1}^n{U_2}^m)U_1-U_2^{-1}({U_1}^n{U_2}^m)={\lambda}^{-n}({\lambda}^{m}U_1U_2-1)U_1^{n}U_2^{m-1}$
\end{center}
and
\begin{center}
$b_1(0,U_1^nU_2^{m-1})= ({U_1}^n{U_2}^{m-1})U_2-U_1({U_1}^n{U_2}^{m-1})= U_1^nU_2^{m}(U_2-\lambda^{-n}U_1).$
\end{center}
From the above relations it is clear that in the quotient space, we have only the coefficients $a_{0,0} \text{ and } a_{1,0}$ remaining. Hence we now have the following result. 
\begin{center}
$H_0(\mathcal A_\theta^{alg},{}_{i}\mathcal A^{alg}_\theta)= {}_{i}\mathcal A^{alg}_\theta/ \langle{\lambda}^{-n}({\lambda}^{m}U_1U_2-1)U_1^{n}U_2^{m-1}, U_1^nU_2^{m}(U_2-\lambda^{-n}U_1) \rangle $
\end{center}
Now to obtain $(H_0(\mathcal A_\theta^{alg}, {}_{i}\mathcal A_\theta^{alg}))^{\mathbb Z_4}$ we consider complex map $h: J_\ast \to C_\ast$, since the map $h_0 : J_0(\mathcal A_\theta^{alg}) \to C_0(\mathcal A_\theta^{alg})$ is the identity map, the $\mathbb Z_4$ action on the bar complex is translated to the Kozul complex with no alteration. Hence we get that
\begin{center}
$HH_0((\mathcal {A}_{\theta, i}^{alg \bullet})^{\mathbb {Z}_4})=  \langle \bar{a_{0,0}}, \bar{a_{1,0}} \rangle$.
\end{center}
\end{proof}

\begin{lemma} $H_1(\mathcal A_\theta^{alg},{}_{i}\mathcal A^{alg}_\theta) \cong 0$.
\end{lemma}
\begin{proof}
 From the previous calculations we do have an explicit formula for the kernel and the image equations of $H_1(\mathcal A_\theta^{alg}, {}_{i}\mathcal A_\theta^{alg})$.
\begin{center}
$\lambda^{-m}a^1_{n-1,m}-\lambda^{n}a^1_{n,m+1}=a^2_{n-1,m}-a^2_{n,m-1}$ ( Kernel Equation )
\end{center}
\begin{center}
$a^1_{n,m}=a_{n-1,m}-\lambda a_{n,m-1} \text{ and } a^2_{n,m}=\lambda^{-m} a_{n-1,m}-\lambda^{1+n} a_{n,m+1}$ (Image Solution)
\end{center}
As before, \emph{a,b,c, and d} are said to be connected if there exists a kernel equation containing them. So, any given kernel solution will have its elements connected by the kernel diagram. \newline
e.g. equation $\lambda^{-0}a^1_{-1,0}-\lambda^{0}a^1_{0,1}=a^2_{-1,0}-a^2_{0,-1}$ is drawn below: 
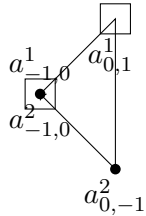
\begin{center}
\begin{tikzpicture}[->]
\draw (-1,0)node[above]{$a^1_{-1,0}$} -- (0,1)node[midway,right]{$a^1_{0,1}$}-- (0,-1)node[below]{$a^2_{0,-1}$}--(-1,0)node[below]{$a^2_{-1,0}$};
\node at (0,1) [transition]{};
\node at (-1,0) [transition]{};
\fill (-1,0) circle (2pt);
\fill (0,-1) circle (2pt);
\end{tikzpicture}
\end{center}

Now, the proof will proceed with induction over the \emph{number} of non-zero elements in a given kernel solution.\par
After going through the process as in the previous cases, one figures out that there is no kernel solution with the number of  non-zero entries less than or equal to 3. Assume that all kernel solutions having number of  non-zero elements less than or equal to $(x-1)$ come from image. Then consider a kernel solution $S_0$ with $x$ non-zero elements in it. Since this solution is finitely supported over the lattice plane, there exists a closed square region $\beta$  over which $S_0$ is supported.
We shall now construct a new solution $S_1$ from $S_0$, with the number of  non-zero elements in $S_1$ at most equal to $x$.

\par
Inside $\beta$, consider the left most column at least one point of which is non-zero. Choose the bottom point $\mu$ of this column. It is clear that $\mu = a^2_{r,s}$ for some $(r,s) \in \mathbb Z^2$. As if it were $a^1_{r,s}$ then the following kernel solution,
\begin{center}
$\lambda^{1-s}a^1_{r-1,s-1}-\lambda^{r}a^1_{r,s}=a^2_{r-1,s-1}-a^2_{r,s-2}$
\end{center}
As, in it all but $a^1_{r,s}$ are zero. This is a contradiction. \par

Now a new solution $S_1$ shall be constructed from $S_0$, with the number of elements at most equal to $x$. To do so,  consider the following map $\wedge$
\begin{center}
$\wedge: \mathbb Z^2 \to \mathbb Z^2 , \text{ such that }$
\end{center}
$(a^1_{r,s+2})^\wedge=(a^2_{r,s})^\wedge=0$. \par
And let $(a^1_{r+1,s+1})^\wedge=a^1_{r+1,s+1}+\lambda^{-1-r}a^2_{r,s} ; (a^2_{r+1,s+1})^\wedge=a^2_{r+1,s+1}+\lambda^{-2-s}a^2_{r,s}$ \newline and  $(a^j_{p,q})^\wedge=a^j_{p,q}$ for other lattice points.

\begin{lemma}Let $S_1:=\wedge(S_0)$. Then $S_1$ is a kernel solution such that $|S_1| \leq |S_0|$. If for some $u \in  {}_{\omega}\mathcal A_\theta^{alg}$, $b_2(u)=S_1$ then $b_2(u')=S_0$, where $u' = u+ {}_{r,s}g$,
\begin{center}
${}_{r,s}g_{p,q}=\begin{cases}
{\lambda}^{-1}a^1_{r,s+2} &\text{ if } (p,q)=(r,s+1)\\
0 & else. \end{cases}$
\end{center}
\end{lemma}
\begin{proof}
Second part is clear that the cardinality of non-zero entries is reduced by 2 when it is asked that $a'^1_{r,s+2}=a'^2_{r,s}=0$ but then the cardinality will then increase by two in case both $a^1_{r+1,s} \text{ and } a^2_{r+1,s}$ are zero in $S_0$. \newline
To check that $S_1$ is a solution, it is enough to check that the kernel equations containing any of these four altered elements hold in $S_1$.
That is,
\begin{center}
$\lambda^{-m}a^1_{n-1,m}-\lambda^{n}a^1_{n,m+1}=a^2_{n-1,m}-a^2_{n,m-1}$ holds
\end{center}
for $(n,m)=(r,s+1), (r+1,s+2), (r+1,s) \text{ and } (r+2,s+1)$.\newline
\underline{Case 1: $(n,m)=(r,s+1)$} \newline
$\lambda^{-1-s}(a^1_{r-1,s+1})^\wedge-\lambda^{r}(a^1_{r,s+2})^\wedge=(a^2_{r-1,s+1})^\wedge-(a^2_{r,s})^\wedge$ \newline
 $\implies \lambda^{-1-s}a^1_{r-1,s+1}-0=a^2_{r-1,s+1}-0.$\newline
This holds as : $a^1_{r-1,s+1}=a^2_{r-1,s+1}=0$.\newline
\underline{Case 2: $(n,m)=(r+1,s+2)$} \newline
$\lambda^{-2-s}(a^1_{r,s+2})^\wedge-\lambda^{r+1}(a^1_{r+1,s+3})^\wedge=(a^2_{r,s+2})^\wedge-(a^2_{r+1,s+1})^\wedge$ \newline
$ \implies 0-\lambda^{r+1}a^1_{r+1,s+3}=a^2_{r,s+2}-a^2_{r+1,s+1}-\lambda^{-2-s} a^1_{r,s+2}$\newline
$\implies \lambda^{-2-s}a^1_{r,s+2}-\lambda^{r+1}a^1_{r+1,s+3}=a^2_{r,s+2}-a^2_{r+1,s+1}$.\newline
This holds as $S_0$ was a kernel solution.\newline
\underline{Case 3: $(n,m)=(r+1,s)$} \newline
$\lambda^{-s}(a^1_{r,s})^\wedge-\lambda^{r+1}(a^1_{r+1,s+1})^\wedge=(a^2_{r,s})^\wedge-(a^2_{r+1,s-1})^\wedge$ \newline
$ \implies\lambda^{-s}a^1_{r,s}-\lambda^{r+1}(a^1_{r+1,s+1}+\lambda^{-1-r}a^2_{r,s})=0-a^2_{r+1,s-1}$ \newline
$\implies \lambda^{-s}a^1_{r,s}-\lambda^{r+1}a^1_{r+1,s+1}=a^2_{r,s}-a^2_{r+1,s-1}.$\newline
This holds as $S_0$ was a kernel solution.\newline
\underline{Case 4: $(n,m)=(r+2,s+1)$} \newline
$\lambda^{-1-s}(a^1_{r+1,s+1})^\wedge-\lambda^{r+2}(a^1_{r+2,s+2})^\wedge=(a^2_{r+1,s+1})^\wedge-(a^2_{r+2,s})^\wedge$ \newline
$\lambda^{-1-s}(a^1_{r+1,s+1}+\lambda^{-1-r}a^2_{r,s})-\lambda^{r+2}a^1_{r+2,s+2}=a^2_{r+1,s+1}+\lambda^{-2-s} a^1_{r,s+2}-a^2_{r+2,s}$\newline
$\lambda^{-1-s}a^1_{r+1,s+1}+\lambda^{-2-s-r}a^2_{r,s}-\lambda^{r+2}a^1_{r+2,s+2}=a^2_{r+1,s+1}+\lambda^{-2-s}a^1_{r,s+2}-a^2_{r+2,s}$\newline
Using $a^1_{r,s+2} = \lambda^{-r} a^2_{r,s}$, we can reduce the above equation to the following equality.
\begin{center}
$\lambda^{-1-s}a^1_{r+1,s+1}-\lambda^{r+2}a^1_{r+2,s+2}=a^2_{r+1,s+1}-a^2_{r+2,s}$
\end{center}
This holds as $S_0$ was a kernel solution.
\newline
It is easy to see that like kernel diagram we can construct image diagrams. These diagrams are obtained by looking at how $b_2$-image of a non-zero lattice point look like as an element of $\mathcal A_\theta^{alg} \oplus \mathcal A_\theta^{alg}$.
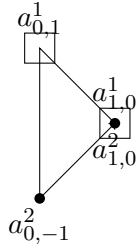
\begin{center}
\begin{tikzpicture}[->]
\draw (1,0)node[right,above,above]{$a^1_{1,0}$} -- (0,1)node[above]{$a^1_{0,1}$}-- (0,-1)node[below]{$a^2_{0,-1}$}--(1,0)node[below]{$a^2_{1,0}$};
\node at (0,1) [transition]{};
\node at (1,0) [transition]{};
\fill (1,0) circle (2pt);
\fill (0,-1) circle (2pt);
\end{tikzpicture}
\end{center}
So, it is clear from the diagram as well as from the image solution equations an image element induces kernel solution elements to its right and up for $a^1$, or right and below for $a^2$ elements. Hence the plausible image elements for the elements of $S_1$ other than the four altered ones are not altered and hence $(a_1,a_2)$ at these lattice point come from image elements by induction.

So, it is clear from the diagram as well as from the explicit equations that we have for the map $b_2$, an image element $a_{0,0}$ induces kernel solution elements to its right-below and up for $a^1$, or right and below for $a^2$ elements. Hence if we verify that
\begin{center}
${b_2(u')}_{p,q} = (a^1_{p,q},a^2_{p,q}) \text{ for } (p,q) = (r,s+2), (r,s), (r+1,s+1)$
\end{center}
then we have proved that $b_2(u') = S_0$. \par
Observe that \newline
$(b_2(u'))_{r,s+2}=(b_2(u+{}_{rs}g))_{r,s+2}=(b_2(u))_{r,s+2} + (b_2({}_{rs}g))_{r,s+2} = ((a^1_{r,s+2})^\wedge,(a^2_{r,s+2})^\wedge)+(a^1_{r,s+2},0) = (a^1_{r,s+2}, a^2_{r,s+2})$ ,\newline
$(b_2(u'))_{r,s}=(b_2(u+{}_{rs}g))_{r,s}=(b_2(u))_{r,s} + (b_2({}_{rs}g))_{r,s} = ((a^1_{r,s})^\wedge,(a^2_{r,s})^\wedge)+(0,\lambda^r a^1_{r,s+2}) = (a^1_{r,s}, a^2_{r,s})$, and \newline
$(b_2(u'))_{r+1,s+1}=(b_2(u+{}_{rs}g))_{r+1,s+1}=(b_2(u))_{r+1,s+1} + (b_2({}_{rs}g))_{r+1,s+1} \newline= ((a^1_{r+1,s+1})^\wedge,(a^2_{r+1,s+1})^\wedge)+(-\lambda^{-1-r}a^2_{r,s},-\lambda^{-s-2}a^1_{r,s}) = (a^1_{r+1,s+1}, a^2_{r+1,s+1}).$ \newline
\end{proof}
\begin{lemma}
 $\wedge^N(S_0) = 0$ for any solution $S_0$ for a sufficiently large number $N$.
\end{lemma}
\begin{proof}
The lemma states that $|S_1| \leq |S_0|$ is strict after sufficiently many iterations. Notice that the above process reduces the number of non-zero entries in the left most column by 2. And when this process is iterated this will lead to shifting of the left most non-zero column rightwards. Pictorially if it is assumed that at no stage the iteration reduces the number of non-zero elements, then these many elements are compressed by the shifting of the left most non-zero column towards right by this process as the right edge of $\beta$ is not changed by this process unless all the non-zero elements are collected on this right edge of $\beta$. \par
Hence, all the non-zero elements of $\wedge^d(S_0)$ will at some stage be bounded to the right edge of $\beta$. But, then, consider the lowest element of $\wedge^d(S_0)$ on this edge. It has to be only $a^2_{w,t}$ for some $(w,t)$. Consider the following kernel equation,
\begin{center}
$\lambda^{-t}a^1_{w,t}-\lambda^{w+1}a^1_{w+1,t+1}=a^2_{w,t}-a^2_{w+1,t-1}$ .
\end{center}
The above equation has all but $a^2_{w,t}$ as zero, which is a contradiction.
\end{proof}
Hence, with the above lemmas we conclude that 
\begin{center}
$HH_1(\mathcal A_{\theta, i}^{alg})=0$.
\end{center}
\par Above computations were for $g=i \in \mathbb Z_4$. Action of $-i \in \mathbb Z_4$ differs from that of $i$ by swapping $U_1 \text{ and } U_2$. That is action of $i$ on torus generated by $U_2 U_1=\lambda U_1 U_2$ is same as $-i$ acts on torus generated by $U_1U_2=\lambda U_2U_1$. Now, since, the results are independent of $\lambda$, so considering $\lambda^{-1}$ would give the same result. Hence we have the following result,
\begin{center}
 $HH_0(\mathcal {A}_{\theta, \pm i}^{alg \bullet})^{\mathbb Z_4} \cong \mathbb C^2,\text{ } HH_1(\mathcal {A}_{\theta, \pm i}^{alg \bullet}) = 0, \text{ and }HH_2(\mathcal {A}_{\theta, \pm i}^{alg \bullet}) = 0$.
\end{center}
\end{proof}

\begin{thm} The Hochschild homology groups for $\mathcal A_\theta^{alg} \rtimes \mathbb Z_4$ are as follows
\begin{center}
$HH_0(\mathcal A_\theta^{alg} \rtimes \mathbb Z_4) \cong \mathbb C^{8}; \text{ }
HH_1(\mathcal A_\theta^{alg} \rtimes \mathbb Z_4) \cong 0; \text{ }
HH_2(\mathcal A_\theta^{alg} \rtimes \mathbb Z_4) \cong \mathbb C$. 
\end{center}
\end{thm}
\begin{proof} 
We know that $HH_0(\mathcal {A}_{\theta, \pm i}^{alg}) \cong \mathbb C^2$, following the process of checking the invariance of a cycle, we obtain that $HH_0(\mathcal {A}_{\theta, \pm i}^{alg})^{\mathbb Z_4} \cong \mathbb C^2$. Similarly we compute that $HH_0(\mathcal {A}_{\theta, -1}^{alg})^{\mathbb Z_4} \cong \mathbb C^3$. To see this we consider the action of $i$ on the elements of $HH_0(\mathcal {A}_{\theta, -1}^{alg})$, we observe that under this action
\begin{enumerate}
\item [$\bullet$] $1 \mapsto 1$, 

\item [$\bullet$] $U_1 \mapsto U_2^{-1}\sim U_2$,

\item [$\bullet$] $U_2 \mapsto U_1$,

\item [$\bullet$] $U_1U_2 \mapsto U_2^{-1}U_1 \sim U_1U_2.$.
\end{enumerate}
Hence we see that 
\begin{center}
$\varphi =a1+bU_1+cU_2+dU_1U_2 \mapsto a1+bU_2+cU_1+dU_1U_2$.
\end{center} 
The above element is invariant iff $b=c$ hence we have a $3$ dimensional invariant sub-space of $HH_0(\mathcal {A}_{\theta, -1}^{alg})$. Hence, we obtain that 
\begin{center}
$HH_0(\mathcal A_\theta^{alg} \rtimes \mathbb Z_4) \cong \mathbb C^{8}$. 
\end{center}
Since we have $HH_1(\mathcal A_\theta^{alg}, {}_{\pm i}\mathcal A_\theta^{alg})=HH_1(\mathcal A_\theta^{alg}, {}_{-1}\mathcal A_\theta^{alg}) =HH_1(\mathcal A_\theta^{alg}, \mathcal A_\theta^{alg})^{\mathbb Z_4}=0$, we conclude that 
\begin{center}
$HH_1(\mathcal A_\theta^{alg} \rtimes \mathbb Z_4) \cong 0$.
\end{center}
As, $HH_2(\mathcal {A}_{\theta, \pm i}^{alg}) = HH_2(\mathcal {A}_{\theta, -1}^{alg}) = 0$, while $HH_2(\mathcal {A}_{\theta}^{alg})^{\mathbb Z_4} \cong \mathbb C$. We have 
\begin{center}
$HH_2(\mathcal A_\theta^{alg} \rtimes \mathbb Z_4) \cong \mathbb C$.
\end{center}
\end{proof}

\subsection{Cyclic homology of  $\mathcal A_\theta^{alg} \rtimes \mathbb Z_4$}.\newline
\begin{thm}
$HC_{even}(\mathcal A_{\theta, \pm i}^{alg})^{\mathbb Z_4} \cong \mathbb C^2$, while 
$HC_{odd}(\mathcal A_{\theta, \pm i}^{alg})^{\mathbb Z_4} = 0.$
\end{thm}
\begin{proof}
We apply the $S,B,I$ long exact sequence relating the Hochschild and cyclic homology of an algebra $A$.
\begin{center}
$... \xrightarrow{B} HH_n(A) \xrightarrow{I} HC_n(A) \xrightarrow{S} HC_{n-2}(A) \xrightarrow{B} HH_{n-1}(A) \xrightarrow{I}...$.\\
\end{center}
Since the $\mathbb Z_4$ action on $\mathcal A_{\theta, {\pm i}}^{alg}$ commutes with the map ${}_{{\pm i}}b$ , we obtain the following exact sequence
\begin{center}
$... \xrightarrow{B} (HH_n(\mathcal A^{alg}_{\theta,{\pm i}}))^{\mathbb Z_4} \xrightarrow{I} (HC_n(\mathcal A^{alg}_{\theta,{\pm i}}))^{\mathbb Z_4} \xrightarrow{S} (HC_{n-2}(\mathcal A^{alg}_{\theta, {\pm i}}))^{\mathbb Z_4} \xrightarrow{B} (HH_{n-1}(\mathcal A^{alg}_{\theta,{\pm i}}))^{\mathbb Z_4} \xrightarrow{I}...$.
\end{center}
\par
We know that $(HH_2(\mathcal A^{alg}_{\theta,{\pm i}}))^{\mathbb Z_4}=(HH_1(\mathcal A^{alg}_{\theta,{\pm i}}))^{\mathbb Z_4}=0$. Hence we obtain that
\begin{center}
$(HC_2(\mathcal A^{alg}_{\theta,{\pm i}}))^{\mathbb Z_4} \cong (HC_0(\mathcal A^{alg}_{\theta, {\pm i}}))^{\mathbb Z_4}$.
\end{center}
But, a preliminary result shows that $HH_0(\mathcal A_{\theta,{\pm i}}^{alg})=HC_0(\mathcal A_{\theta,{\pm i}}^{alg})$. Hence, we obtain that
\begin{center}
 $(HC_2(\mathcal A^{alg}_{\theta,{\pm i}}))^{\mathbb Z_4}=(HC_0(\mathcal A^{alg}_{\theta,{\pm i}}))^{\mathbb Z_4} \cong \mathbb C^2$.
\end{center}
Also, since $(HH_2(\mathcal A^{alg}_{\theta,\omega^{\pm 1}}))^{\mathbb Z_4}=(HH_3(\mathcal A^{alg}_{\theta,\pm i}))^{\mathbb Z_4}=0$, we obtain that
\begin{center}
$(HC_3(\mathcal A^{alg}_{\theta,\pm i}))^{\mathbb Z_4}\cong (HC_1(\mathcal A^{alg}_{\theta,\pm i}))^{\mathbb Z_4}$.
\end{center}
 Since $HH_1(\mathcal A_{\theta,\pm i}^{alg})=0$, we have $(HC_1(\mathcal A^{alg}_{\theta,\pm i}))^{\mathbb Z_4}=0$. Hence, we have 
\begin{center}
$HC_{odd}(\mathcal A_{\theta, \pm i}^{alg})^{\mathbb Z_4} =0$.
\end{center}
\end{proof}

\subsection{Periodic cyclic homology}.\newline
\begin{thm}
$HC_{even}(\mathcal A_\theta^{alg} \rtimes \mathbb Z_4) \cong \mathbb C^{9} \text{ while } HC_{odd}(\mathcal A_\theta^{alg} \rtimes \mathbb Z_4) = 0$.
\end{thm}
\begin{proof}
We have.
\begin{center}
$... \xrightarrow{B} (HH_2(\mathcal A^{alg}_{\theta}))^{\mathbb Z_4} \xrightarrow{I} (HC_2(\mathcal A^{alg}_{\theta}))^{\mathbb Z_4} \xrightarrow{S} (HC_{0}(\mathcal A^{alg}_{\theta}))^{\mathbb Z_4} \xrightarrow{B} (HH_{1}(\mathcal A^{alg}_{\theta,}))^{\mathbb Z_4} \xrightarrow{I}...$
\end{center}
Since $HH_2(\mathcal A_\theta^{alg})^{\mathbb Z_4} \cong HC_0(\mathcal A_\theta^{alg})^{\mathbb Z_4} \cong \mathbb C$, we have $(HC_2(\mathcal A_\theta^{alg}))^{\mathbb Z_4} \cong \mathbb C^2$. Also we notice here that the $\mathbb Z_4$ invariant sub-space of $HC_{even}(\mathcal A_{\theta, -1}^{alg})$ is $3$ dimensional. Hence we conclude that 
\begin{center}
$HC_{even}(\mathcal A_\theta^{alg} \rtimes \mathbb Z_4) \cong \mathbb C^{9}$.
\end{center} \par
As for the odd cyclic homology, we have $(HC_3(\mathcal A^{alg}_{\theta,\pm i}))^{\mathbb Z_4}\cong (HC_1(\mathcal A^{alg}_{\theta,\pm i}))^{\mathbb Z_4}=0$, and we also have $HC_3(\mathcal A_\theta^{alg})^{\mathbb Z_4} = HC_1(\mathcal A_\theta^{alg})^{\mathbb Z_4}=HH_1(\mathcal A_\theta^{alg})^{\mathbb Z_4} = 0$. Combining these two results, we obtain that
\begin{center}
$HC_{odd}(\mathcal A_\theta^{alg} \rtimes \mathbb Z_4) = 0$.
\end{center}
\par So we have computed the Hochschild and cyclic homology of the $\mathbb Z_4$ orbifold.
\end{proof}


\section{$\mathbb Z_6$ action on $\mathcal A_\theta^{alg}$}
The group $\mathbb Z_6$ is embedded in $SL(2,\mathbb Z)$ through its generator $g= \left[
 \begin{array}{cc}
   0 & -1 \\
   1& 1
 \end{array} \right]\in SL(2,\mathbb Z)$. The generator acts of $\mathcal A_\theta^{alg}$ in the following way
\begin{center}
$U_1 \mapsto U_2, U_2 \mapsto \displaystyle \frac{U_1^{-1} U_2}{\sqrt \lambda}$.
\end{center}
\subsection{Hochschild homology}.\newline
We will use $-\omega$ to stand for the generator of $\mathbb Z_6$.
\begin{thm}
$HH_0((\mathcal {A}_{\theta, {-\omega}}^{alg \bullet})^{\mathbb {Z}_6}) \cong \mathbb C$, while $HH_k((\mathcal {A}_{\theta, {-\omega}}^{alg \bullet})^{\mathbb {Z}_6})$  are trivial groups for $k \geq 1$.
\end{thm}

\begin{lemma}
Consider the following chain complex $J_{\ast, {-\omega}}^{\mathbb Z_6}$
\begin{center}
$J_{\ast ,  {-\omega}}^{\mathbb Z_6} := 0 \xleftarrow{{}_{-\omega}b} (\mathcal A^{alg}_{\theta, {-\omega}})^{\mathbb Z_6} \xleftarrow{{}_{-\omega}b} ((\mathcal A^{alg}_{\theta,  {-\omega}})^{\otimes 2})^{\mathbb Z_6} \xleftarrow{{}_{-\omega}b} ((\mathcal A^{alg}_{\theta,  {-\omega}})^{\otimes 3})^{\mathbb Z_6}...$
\end{center}
where,
\begin{center}
${}_{-\omega}b(a_0\otimes a_1\otimes....\otimes a_n)= b'(a_0\otimes a_1\otimes....\otimes a_n)+(-1)^n(( {-\omega}^2\cdot a_n)a_0\otimes a_1\otimes....\otimes a_{n-1})$.
\end{center}
Then,
\begin{center}
$H_\bullet(J_{\ast , {-\omega}}^{\mathbb Z_6}, {}_ {-\omega}b) = (H_\bullet(J_\ast(\mathcal A_\theta^{alg}, {}_{-\omega}\mathcal A_\theta^{alg})), b)^{\mathbb Z_6}$
\end{center}
\end{lemma}

Hence using the adjusted Connes' complex for the algebraic case we can now calculate the homology groups.
\begin{center}
\begin{enumerate}
\item[$\bullet$]$H_0(\mathcal A_\theta^{alg},{}_{-\omega}\mathcal A^{alg}_\theta)= {}_{-\omega}\mathcal A^{alg}_\theta\otimes_{\mathfrak B_\theta^{alg}} \mathfrak B_\theta^{alg} / Image(1\otimes b_1),$
\item[$\bullet$]$H_1(\mathcal A_\theta^{alg},{}_{-\omega}\mathcal A^{alg}_\theta)= Ker(1\otimes b_1) / Image(1\otimes b_2),$
\item[$\bullet$]$H_2(\mathcal A_\theta^{alg},{}_{-\omega}\mathcal A^{alg}_\theta)= Ker(1\otimes b_2)$.
\end{enumerate}
\end{center}

\begin{lemma} $H_2(\mathcal A_\theta^{alg},{}_{-\omega}\mathcal A^{alg}_\theta) \cong 0.$
\end{lemma}
\begin{proof}
We prove here for case $g= {-\omega}$, the proof for $g={-\omega}^2$ is similar. Consider the map $(1 \otimes b_2)$ in the tensor complex. To calculate the kernel of this map we have a closer look at this map,
\begin{center}
$(1\otimes b_2)(a\otimes I)=a\otimes_{\mathfrak B_\theta^{alg}}(U_2\otimes I-\lambda\otimes U_2)\otimes e_1-a\otimes_{\mathfrak B_\theta^{alg}}(\lambda U_1\otimes I-I\otimes U_1)\otimes e_2$.
\end{center}
Using the twisted bimodule structure of ${}_{-\omega}\mathcal A_\theta^{alg}$ over $\mathcal A_\theta^{alg}$, we simplify the equation to the following,
\begin{center}
$(1\otimes b_2)(a\otimes I)= ( a U_2-\displaystyle\frac{ U_1^{-1}U_2}{\sqrt\lambda}a, U_2a-\lambda aU_1)$.
\end{center}
Hence we obtain the following relation over an element $(a \otimes 1)$ to reside in $ker( 1\otimes b_2)$. 
\begin{center}
$H_2(\mathcal A_\theta^{alg},{}_{-\omega}\mathcal A^{alg}_\theta)= \left\lbrace a\in {}_{-\omega}\mathcal A_\theta^{alg} | {\lambda}^{n+1.5} a_{n+1,m}=a_{n,m-1}; {\lambda}^{m-n}a_{n-1,m}=\lambda a_{n,m-1}\right\rbrace$ .
\end{center}
Since, no such nontrivial element exists in ${}_{-\omega}\mathcal A_\theta^{alg}$ because if it did then $|a _{n,m}|=|a_{n+2,m}|$. So, we have the desired result.
\begin{center}
$H_2(\mathcal A_\theta^{alg},{}_{-\omega}\mathcal A^{alg}_\theta)=0$.
\end{center}
\end{proof}

\begin{lemma} $H_0(\mathcal A_\theta^{alg},{}_{-\omega}\mathcal A^{alg}_\theta)^{\mathbb Z_6} \cong \mathbb C.$
\end{lemma}
\begin{proof}
We have the map $(1 \otimes b_1)$ in the tensor complex defined below,
\begin{center}
$(1\otimes b_1)(a\otimes I\otimes e_j)=a\otimes_{\mathfrak B_\theta^{alg}}(I\otimes U_j-U_j\otimes I)= U_j^{-1}a-aU_j$.
\end{center}
As before if we use the twisted bicomplex structure of ${}_{-\omega}\mathcal A_\theta^{alg}$ over the algebra $\mathcal A_\theta^{alg}$, the map$(1 \otimes b_1)$ can be simplified as follows,
\begin{center}
$b_1(a_1,0)= a_1U_1- U_2a_1\text{ and } b_1(0,a_2)= a_2 U_2-{\lambda}^{-0.5}U_1^{-1}U_2a_2$.
\end{center}
Observe that  $b_1(a_1,a_2)= b_1(a_1,0)+b_1(0,a_2)$. Further we can simplify the calculation by considering only elements of the type $b_1(U_1^nU_2^m,0)$, and similarly for $b_2(0,a_2)$ we can consider the elements of the type $b_2(0,U_1^nU_2^m)$.
We observe that 
\begin{center}
$b_1(U_1^nU_2^m,0)= ({U_1}^n{U_2}^m)U_1-U_2({U_1}^n{U_2}^m)=({\lambda}^{m}U_1-U_2)U_1^{n}U_2^{m}$
\end{center}
and
\begin{center}
$b_2(0,U_1^nU_2^{m-1})= ({U_1}^n{U_2}^{m-1})U_2-{\lambda}^{-0.5}U_1^{-1}U_2({U_1}^n{U_2}^{m-1})= (U_1-{\lambda}^{n-0.5})U_1^{n-1}U_2^{m}.$
\end{center}
From the above relations it is clear that in the quotient space, we have only the coefficients $a_{0,0}$ remains. Hence we now have the following result. 
\begin{center}
$H_0(\mathcal A_\theta^{alg},{}_{-\omega}\mathcal A^{alg}_\theta)= {}_{-\omega}\mathcal A^{alg}_\theta / \langle{\lambda}^{-n}({\lambda}^{m}U_1-{\lambda}^nU_2)U_1^{n}U_2^{m}, (U_1-{\lambda}^{n-0.5})U_1^{n-1}U_2^{m} \rangle $
\end{center}
Now to obtain $(H_0(\mathcal A_\theta^{alg}, {}_{-\omega}\mathcal A_\theta^{alg}))^{\mathbb Z_6}$ we consider complex map $h: J_\ast \to C_\ast$, since the map $h_0 : J_0(\mathcal A_\theta^{alg}) \to C_0(\mathcal A_\theta^{alg})$ is the identity map, the $\mathbb Z_6$ action on the bar complex is translated to the Kozul complex with no alteration. Hence the equivalence class of $\bar{a_{0,0}}$ is invariant under $\mathbb Z_6$ action. So we have 
\begin{center}
$HH_0((\mathcal {A}_{\theta, {-\omega}}^{alg})^{\mathbb {Z}_6})= \langle \bar{a_{0,0}} \rangle$
\end{center}
\end{proof}

\begin{lemma} $H_1(\mathcal A_\theta^{alg},{}_{-\omega}\mathcal A^{alg}_\theta) \cong 0$.
\end{lemma}
\begin{proof}
 From the previous calculations we do have an explicit formula for the kernel and the image equations of $H_1(\mathcal A_\theta^{alg}, {}_{-\omega}\mathcal A_\theta^{alg})$,
\begin{center}
$\lambda^{m}a^1_{n-1,m}-\lambda^{n}a^1_{n,m-1}-{\lambda}^{n+0.5}a^2_{n+1,m-1}+a^2_{n,m-1}=0$ ( Kernel Equation ),
\end{center}
\begin{center}
$a^1_{n,m}={\lambda}^{n+1.5} a_{n+1,m-1}-a_{n,m-1} \text{ and } a^2_{n,m}=\lambda^{1+m} a_{n-1,m}-\lambda^{n} a_{n,m-1}$ (Image Solution).
\end{center}
As before, \emph{a,b,c, and d} are said to be connected if there exists a kernel equation containing them. So, any given kernel solution will have its elements connected by the kernel diagram. \newline
e.g. equation $\lambda^{0}a^1_{-1,0}-\lambda^{0}a^1_{0,-1}-{\sqrt\lambda}a^2_{1,-1}+a^2_{0,-1}$ is drawn below: 
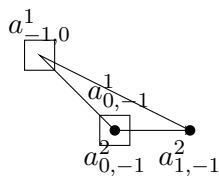
\begin{center}
\begin{tikzpicture}[->]
\draw (1,-1)--(-1,0)node[above]{$a^1_{-1,0}$} -- (0,-1)node[midway,right]{$a^1_{0,-1}$}-- (0,-1)node[below]{$a^2_{0,-1}$}--(1,-1)node[below]{$a^2_{1,-1}$};
\node at (-1,0) [transition]{};
\node at (0,-1) [transition]{};
\fill (1,-1) circle (2pt);
\fill (0,-1) circle (2pt);
\end{tikzpicture}
\end{center}

Now, the proof will proceed with induction over the \emph{number} of non-zero elements in a given kernel solution.\par
After going through the process, explicitly described in previous cases, we conclude that there are no kernel solution with the number of  non-zero entries less than or equal to 3. Assume that all kernel solutions having number of  non-zero elements less than or equal to $(x-1)$ come from image. Then consider a kernel solution $S_0$ with $x$ non-zero elements in it. Since this solution is finitely supported over the lattice plane, there exists a closed square region $\beta$  over which $S_0$ is supported.
We shall now construct a new solution $S_1$ from $S_0$, with the number of  non-zero elements in $S_1$ at most equal to $x$.

\par
Inside $\beta$, consider the left most column at least one point of which is non-zero. Choose the bottom point $\mu$ of this column. It is clear that $\mu = a^1_{r,s}$ for some $(r,s) \in \mathbb Z^2$. As if it were $a^2_{r,s}$ then the following kernel solution.
\begin{center}
$\lambda^{1+s}a^1_{r-2,s+1}-\lambda^{r-1}a^1_{r-1,s}-{\lambda}^{r-0.5}a^2_{r,s}+a^2_{r-1,s}=0$
\end{center}
As all terms in it but $a^2_{r,s}$ are zero, we derive a contradiction. \par

Now a new solution $S_1$ shall be constructed from $S_0$, with the number of  elements at most equal to $x$. To do so,  consider the following map $\wedge$
\begin{center}
$\wedge: \mathbb Z^2 \to \mathbb Z^2 , \text{ such that }$
\end{center}
$(a^1_{r,s})^\wedge=0; \newline
(a^1_{r+1,s})^\wedge=a^1_{r+1,s}+\lambda^{-r-1.5} a^1_{r,s}; \newline (a^2_{r+1,s})^\wedge=a^2_{r+1,s}+\lambda^{-0.5}a^1_{r,s};\newline (a^2_{r+2,s-1})^\wedge = a^2_{r+2,s-1}-{\lambda}^{s-r-1.5}a^1_{r,s}$ and  $(a^j_{p,q})^\wedge=a^j_{p,q}$ for other lattice points.

\begin{lemma}Let $S_1:=\wedge(S_0)$. Then $S_1$ is a kernel solution. If for some $u \in  {}_{-\omega}\mathcal A_\theta^{alg}$, $b_2(u)=S_1$ then $b_2(u')=S_0$, where $u' = u+ {}_{r,s}g$,
\begin{center}
${}_{r,s}g_{p,q}=\begin{cases}
{\lambda}^{-r-1.5}a^1_{r,s} &\text{ if } (p,q)=(r+1,s-1),\\
0 & else. \end{cases}$
\end{center}
\end{lemma}
\begin{proof}
To check that $S_1$ is a solution, it is enough to check that the kernel equations containing any of these four altered elements hold in $S_1$.
That is,
\begin{center}
$\lambda^{m}a^1_{n-1,m}-\lambda^{n}a^1_{n,m-1}-{\lambda}^{n+0.5}a^2_{n+1,m-1}+a^2_{n,m-1}=0$ holds
\end{center}
for $(n,m)=(r,s+1), (r+1,s+1), (r+2,s) \text{ and } (r+1,s)$ \newline
\underline{Case 1: $(n,m)=(r,s+1)$} \newline
$\lambda^{s+1}(a^1_{r-1,s+1})^\wedge-\lambda^{r}(a^1_{r,s})^\wedge-{\lambda}^{r+0.5}(a^2_{r+1,s})^\wedge+(a^2_{r,s})^\wedge=0$ \newline
$\implies \lambda^{s+1}a^1_{r-1,s+1}-0-{\lambda}^{r+0.5}(a^2_{r+1,s}+{\lambda}^{-0.5}a^1_{r,s})+a^2_{r,s}=0$\newline
The above equation holds in $S_0$.\newline
\underline{Case 2: $(n,m)=(r+1,s+1)$} \newline
$\lambda^{s+1}(a^1_{r,s+1})^\wedge-\lambda^{r+1}(a^1_{r+1,s})^\wedge-{\lambda}^{r+1.5}(a^2_{r+2,s})^\wedge+(a^2_{r+1,s})^\wedge=0$ \newline
$ \implies \lambda^{s+1}a^1_{r-1,s+1}-\lambda^{r+1}(a^1_{r+1,s}+{\lambda}^{-r-1.5}a^1_{r,s})-{\lambda}^{r+1.5}(a^2_{r+2,s})+a^2_{r+1,s}+\lambda^{-0.5}a^1_{r,s}=0$\newline
This holds in $S_0$.\newline
\underline{Case 3: $(n,m)=(r+2,s)$} \newline
$\lambda^{s}(a^1_{r+1,s})^\wedge-\lambda^{r+2}(a^1_{r+2,s-1})^\wedge-{\lambda}^{r+2.5}(a^2_{r+3,s-1})^\wedge+(a^2_{r+2,s-1})^\wedge=0$ \newline
$\implies \lambda^{s}(a^1_{r+1,s}+ \lambda^{-r-1.5}a^1_{r,s})-\lambda^{r+2}(a^1_{r+2,s-1})-{\lambda}^{r+2.5}(a^2_{r+3,s-1})+(a^2_{r+2,s-1}-{\lambda}^{s-r-1.5}a^1_{r,s})=0$ \newline
Here also the above condition holds as $S_0$ is a kernel solution.\newline
\underline{Case 4: $(n,m)=(r+1,s)$} \newline
$\lambda^{s}(a^1_{r,s})^\wedge-\lambda^{r+1}(a^1_{r+1,s-1})^\wedge-{\lambda}^{r+1.5}(a^2_{r+2,s-1})^\wedge+(a^2_{r+1,s-1})^\wedge=0$ \newline
$\implies 0-\lambda^{r+1}(a^1_{r+1,s-1})-{\lambda}^{r+1.5}(a^2_{r+2,s-1})+(a^2_{r+2,s-1}+{\lambda}^{s}a^1_{r,s})=0$ \newline
This is a kernel condition in $S_0$. Hence we have checked all the possible cases and found the relations to be coherent with the kernel condition.

\par
It is easy to see that like kernel diagram we can construct image diagrams. These diagrams are obtained by looking at how $b_2$-image of a non-zero lattice point look like as an element of $\mathcal A_\theta^{alg} \oplus \mathcal A_\theta^{alg}$.
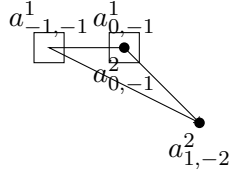
\begin{center}
\begin{tikzpicture}[->]
\draw (1,-2)--(-1,-1)node[above]{$a^1_{-1,-1}$} -- (0,-1)node[above]{$a^1_{0,-1}$}-- (0,-1)node[below]{$a^2_{0,-1}$}--(1,-2)node[below]{$a^2_{1,-2}$};
\node at (-1,-1) [transition]{};
\node at (0,-1) [transition]{};
\fill (1,-2) circle (2pt);
\fill (0,-1) circle (2pt);
\end{tikzpicture}
\end{center}

So, it is clear from the diagram as well as from the explicit equations that we have for the map $b_2$, an image element $a_{0,0}$ induces kernel solution elements to its left and up for $a^1$, and up and right-up for $a^2$ elements. Hence if we verify that
\begin{center}
${b_2(u')}_{p,q} = (a^1_{p,q},a^2_{p,q}) \text{ for } (p,q) = (r+1,s), (r,s), (r+2,s-1)$
\end{center}
then we have proved that $b_2(u') = S_0$. \par
Observe that \newline
$(b_2(u'))_{r,s}=(b_2(u+{}_{rs}g))_{r,s}\newline=(b_2(u))_{r,s} + (b_2({}_{rs}g))_{r,s} = ((a^1_{r,s+2})^\wedge,(a^2_{r,s+2})^\wedge)+({\lambda}^{r+1.5}{\lambda}^{-r-1.5}a^1_{r,s+2},0) = (a^1_{r,s+2}, a^2_{r,s+2})$ ,\newline
$(b_2(u'))_{r+1,s}=(b_2(u+{}_{rs}g))_{r+1,s}\newline=(b_2(u))_{r+1,s} + (b_2({}_{rs}g))_{r+1,s} \newline = ((a^1_{r+1,s})^\wedge,(a^2_{r+1,s})^\wedge)+(-{\lambda}^{-r-1.5}a^1_{r,s},\lambda^{1+r} \lambda^{-r-1.5}a^1_{r,s}) = (a^1_{r+1,s}, a^2_{r+1,s})$, and \newline
$(b_2(u'))_{r+2,s-1}=(b_2(u+{}_{rs}g))_{r+2,s-1}\newline=(b_2(u))_{r+2,s-1} + (b_2({}_{rs}g))_{r+2,s-1} \newline = ((a^1_{r+2,s-1})^\wedge,(a^2_{r+2,s-1})^\wedge)+(0,{\lambda}^{s}\lambda^{-r-1.5}a^1_{r,s}) = (a^1_{r+2,s-1}, a^2_{r+2,s-1})$. \newline
\end{proof}
\begin{lemma}
 $\wedge^N(S_0) = 0$ for any solution $S_0$ for a sufficiently large number $N$.
\end{lemma}
\begin{proof}
We know that from the above lemma that the left-most non-zero column of $S_0$ move rightwards as $\wedge$ transforms the lattice plane $\mathbb Z^2$. While $\wedge$ transforms the lattice plane. We see that the right most column in the region within which lies all the non-zero lattice points of $S_0$, $C_\eta$  does not move rightwards unless $C_{\eta-2}=0$. Consider the solution $S_r:=\wedge^r(S_0)$ such that $C_{\eta-2}=C_{\eta+i}=0$ for $i \geq 1$. Clearly such an $r$ exists. Consider the non-zero element $a^1_{{\eta-1}, w} \in C_{\eta-1}$, such that $a^j_{{\eta-1}, t} = 0$ for $\forall t <w \text{ and } j=1,2$. \par
Now consider the following kernel equation 
\begin{center}
$\lambda^{w}a^1_{{\eta-2},w}-\lambda^{\eta-1}a^1_{{\eta-1},w-1}-{\lambda}^{\eta-0.5}a^2_{{\eta},w-1}+a^2_{{\eta-1},w-1}=0$
\end{center}
We see that $a^2_{\eta,w-1}=0$. \newline
Also similar computation involving elements $a^1_{\eta+1,w-3},a^2_{\eta+1,w-3}, a^2_{\eta+1,w-3}$, and $a^1_{\eta,w-1}$ implies that $a^1_{\eta,w-1}=0$. Now we consider the equation,
\begin{center}
$\lambda^{w}a^1_{{\eta-1},w}-\lambda^{\eta}a^1_{{\eta},w-1}-{\lambda}^{\eta+0.5}a^2_{{\eta+1},w-1}+a^2_{{\eta},w-1}=0$
\end{center}
We see now that $a^1_{\eta-1,w}=0$. This is a contradiction. Hence we have $C_{\eta-1}=0$.\par

Now with a brief glance on the equation,
\begin{center}
$\lambda^{v}a^1_{{\eta},v}-\lambda^{\eta+1}a^1_{{\eta+1},v-1}-{\lambda}^{\eta+1.5}a^2_{{\eta+2},v-1}+a^2_{{\eta+1},v-1}=0$.
\end{center}
With this we conclude that $C_\eta=0$. Hence we arrive at the conclusion that $\wedge^r(S_0) = 0$. Hence, with the above lemmas we conclude that 
\begin{center}
$HH_1(\mathcal A_{\theta, -\omega}^{alg})=0$.
\end{center}
\end{proof}
Above computations were for $g=-\omega \in \mathbb Z_6$. Action of $-\omega^2 \in \mathbb Z_6$ is similar and we leave it to the reader to check that,
\begin{center}
 $HH_0(\mathcal {A}_{\theta, -{\omega}^{\pm 1}}^{alg})^{\mathbb Z_6} \cong \mathbb C,\text{ } HH_1(\mathcal {A}_{\theta, -{\omega}^{\pm 1}}^{alg}) = 0, \text{ and }HH_2(\mathcal {A}_{\theta, -{\omega}^{\pm 1}}^{alg}) = 0$.
\end{center}
\end{proof}


\begin{thm} The Hochschild homology groups for $\mathcal A_\theta^{alg} \rtimes \mathbb Z_6$ are as follows
\begin{center}
$HH_0(\mathcal A_\theta^{alg} \rtimes \mathbb Z_6) \cong \mathbb C^{9}; \text{ }
HH_1(\mathcal A_\theta^{alg} \rtimes \mathbb Z_6) \cong 0; \text{ }
HH_2(\mathcal A_\theta^{alg} \rtimes \mathbb Z_6) \cong \mathbb C$. 
\end{center}
\end{thm}
\begin{proof} We know that $HH_0(\mathcal {A}_{\theta, \omega^{\pm 1}}^{alg \bullet}) \cong \mathbb C^3$, $HH_0(\mathcal {A}_{\theta, -{\omega}^{\pm 1}}^{alg \bullet}) \cong \mathbb C$ and $HH_0(\mathcal {A}_{\theta, -1}^{alg \bullet}) \cong \mathbb C^4$, as we have notice earlier, the $\mathbb Z_6$ action on the zeroth homology is same in both the complexes. Hence we follow the procedure detailed in previous sub-sections to conclude that $HH_0(\mathcal {A}_{\theta, \omega^{\pm 1}}^{alg \bullet})^{\mathbb Z_6} \cong C^3$. To see this we consider the action of $-\omega$ on the elements of $HH_0(\mathcal {A}_{\theta, \omega^{\pm 1}}^{alg})$, we observe that under this action
\begin{enumerate}
\item [$\bullet$] $1 \mapsto 1$, 

\item [$\bullet$] $U_1 \mapsto U_2$,

\item [$\bullet$] $U_2 \mapsto  \displaystyle\frac{U_1^{-1}U_2}{\sqrt\lambda} \sim U_2^{-1} \sim U_1$,

\end{enumerate}
Hence we see that 
\begin{center}
$\varphi =a1+bU_1+cU_2 \mapsto a1+bU_2+cU_1$.
\end{center} 
The above element is invariant iff $b=c$ hence we have a $2$ dimensional invariant sub-space of $HH_0(\mathcal {A}_{\theta, \omega}^{alg})$. \newline
Now we consider the elements of $HH_0(\mathcal {A}_{\theta, -1}^{alg})$ invariant under the action of $g=-\omega$, we observe that under this action

\begin{enumerate}
\item [$\bullet$] $1 \mapsto 1$, 

\item [$\bullet$] $U_1 \mapsto U_2$,

\item [$\bullet$] $U_2 \mapsto \displaystyle\frac{U_2U_1}{\sqrt\lambda} \sim {\sqrt\lambda}{U_1U_2}$,

\item[$\bullet$] $U_1U_2 \mapsto  \displaystyle\frac{U_2U_1^{-1}U_2}{\sqrt\lambda} \sim \displaystyle\frac{U_1^{-1}}{\sqrt\lambda} \sim \displaystyle\frac{U_1}{\sqrt\lambda}$.

\end{enumerate}

Hence we see that 
\begin{center}
$\varphi =a1+bU_1+cU_2 + dU_1U_2\mapsto a1+bU_2+c \sqrt\lambda U_1U_2 + d \displaystyle\frac{U_1}{\sqrt\lambda}$.
\end{center} 
The above element is invariant iff $b=c= \displaystyle\frac{d}{\sqrt\lambda}$ hence we have a $2$ dimensional invariant sub-space of $HH_0(\mathcal {A}_{\theta, -1}^{alg})$. Summing up all the calculated sub-spaces above we conclude that 
\begin{center}
$HH_0(\mathcal A_\theta^{alg} \rtimes \mathbb Z_4) \cong \mathbb C^{9}$.
\end{center}
Previous calculations imply that $HH_1(\mathcal {A}_{\theta, g}^{alg \bullet})^{\mathbb Z_6} =0$ for all $g \in \mathbb Z_6$, hence we have 
\begin{center}
$HH_1(\mathcal A_\theta^{alg} \rtimes \mathbb Z_6) \cong 0$.
\end{center}
We notice that for $1 \neq g \in \mathbb Z_6$, $HH_2(\mathcal {A}_{\theta, g}^{alg \bullet})=0$, while $HH_2(\mathcal {A}_{\theta}^{alg \bullet})^{\mathbb Z_6} \cong \mathbb C$. Hence we have
\begin{center}
$HH_2(\mathcal A_\theta^{alg} \rtimes \mathbb Z_6) \cong \mathbb C$.
\end{center}
\end{proof}

\subsection{Cyclic homology of  $\mathcal A_\theta^{alg} \rtimes \mathbb Z_6$}.\newline
\begin{thm}
$HC_{even}(\mathcal A_{\theta, -\omega}^{alg}) \cong \mathbb C$, while 
$HC_{odd}(\mathcal A_{\theta, -\omega}^{alg}) = 0.$
\end{thm}
\begin{proof}
We apply the $S,B,I$ long exact sequence relating the Hochschild and cyclic homology of an algebra $A$.
\begin{center}
$... \xrightarrow{B} HH_n(A) \xrightarrow{I} HC_n(A) \xrightarrow{S} HC_{n-2}(A) \xrightarrow{B} HH_{n-1}(A) \xrightarrow{I}...$\\.
\end{center}
Since the $\mathbb Z_6$ action on $\mathcal A_{\theta, {-\omega}}^{alg}$ commutes with the map ${}_{{-\omega}}b$ , we obtain the following exact sequence
\begin{center}
$... \xrightarrow{B} (HH_n(\mathcal A^{alg}_{\theta,{-\omega}}))^{\mathbb Z_6} \xrightarrow{I} (HC_n(\mathcal A^{alg}_{\theta,{-\omega}}))^{\mathbb Z_6} \xrightarrow{S} (HC_{n-2}(\mathcal A^{alg}_{\theta, {-\omega}}))^{\mathbb Z_6} \xrightarrow{B} (HH_{n-1}(\mathcal A^{alg}_{\theta,{-\omega}}))^{\mathbb Z_6} \xrightarrow{I}...$.
\end{center}
\par
We know that $(HH_2(\mathcal A^{alg}_{\theta,{-\omega}}))^{\mathbb Z_6}=(HH_1(\mathcal A^{alg}_{\theta,{-\omega}}))^{\mathbb Z_6}=0$. Hence we obtain that
\begin{center}
$(HC_2(\mathcal A^{alg}_{\theta,{-\omega}}))^{\mathbb Z_6} \cong (HC_0(\mathcal A^{alg}_{\theta, {-\omega}}))^{\mathbb Z_6}$.
\end{center}
But, a preliminary result shows that $HH_0(\mathcal A_{\theta,{-\omega}}^{alg})=HC_0(\mathcal A_{\theta,{-\omega}}^{alg})$. Hence, we obtain that
\begin{center}
 $(HC_2(\mathcal A^{alg}_{\theta,{-\omega}}))^{\mathbb Z_6}=(HC_0(\mathcal A^{alg}_{\theta,{-\omega}}))^{\mathbb Z_6} \cong \mathbb C$.
\end{center}
Also, since $(HH_2(\mathcal A^{alg}_{\theta,-\omega}))^{\mathbb Z_6}=(HH_3(\mathcal A^{alg}_{\theta,-\omega}))^{\mathbb Z_6}=0$, we obtain that
\begin{center}
$(HC_3(\mathcal A^{alg}_{\theta,-\omega}))^{\mathbb Z_6}\cong (HC_1(\mathcal A^{alg}_{\theta,-\omega}))^{\mathbb Z_6}$. 
\end{center}
Since, $HH_1(\mathcal A_{\theta,-\omega}^{alg})=0$, we have $(HC_1(\mathcal A^{alg}_{\theta,-\omega}))^{\mathbb Z_6}=0$. Hence we conclude that
\begin{center}
$HC_{odd}(\mathcal A_{\theta, -\omega}^{alg}) =0$.
\end{center}
\end{proof}

\subsection{Periodic cyclic homology}.\newline
\begin{thm}
$HC_{even}(\mathcal A_\theta^{alg} \rtimes \mathbb Z_6) \cong \mathbb C^{10} \text{ while } HC_{odd}(\mathcal A_\theta^{alg} \rtimes \mathbb Z_6) = 0$.
\end{thm}
\begin{proof}
We have the $S,B,I$ sequence relating the Hochschild and the cyclic homology.
\begin{center}
$... \xrightarrow{B} (HH_2(\mathcal A^{alg}_{\theta, -\omega}))^{\mathbb Z_6} \xrightarrow{I} (HC_2(\mathcal A^{alg}_{\theta,-\omega}))^{\mathbb Z_6} \xrightarrow{S} (HC_{0}(\mathcal A^{alg}_{\theta,-\omega}))^{\mathbb Z_6} \xrightarrow{B} (HH_{1}(\mathcal A^{alg}_{\theta,-\omega}))^{\mathbb Z_6} \xrightarrow{I}...$
\end{center}
Hence, $(HC_2(\mathcal A_\theta^{alg}))^{\mathbb Z_6} \cong \mathbb C^2$. Also we notice here that in this case, the $\mathbb Z_6$ invariant sub-space of $HC_{even}(\mathcal A_{\theta, -1}^{alg})$ is $2$ dimensional and the $\mathbb Z_6$ invariant sub-space of $HC_{even}(\mathcal A_{\theta, \pm\omega}^{alg})$ is $2$ dimensional. Hence we conclude that 
\begin{center}
$HC_{even}(\mathcal A_\theta^{alg} \rtimes \mathbb Z_6) \cong \mathbb C^{10}$.
\end{center} \par
As for the odd cyclic homology, we have $(HC_3(\mathcal A^{alg}_{\theta,\pm \omega}))^{\mathbb Z_6}\cong (HC_1(\mathcal A^{alg}_{\theta,\pm \omega}))^{\mathbb Z_6}=0$, and we also have $HC_3(\mathcal A_\theta^{alg})^{\mathbb Z_6} = HC_1(\mathcal A_\theta^{alg})^{\mathbb Z_6}=HH_1(\mathcal A_\theta^{alg})^{\mathbb Z_6} = 0$. Combining these two results, we obtain that 
\begin{center}
$HC_{odd}(\mathcal A_\theta^{alg} \rtimes \mathbb Z_6) = 0$.
\end{center}
\par So we have computed the Hochschild and cyclic homology of the $\mathbb Z_6$ orbifold.
\end{proof}

\section{Hochschild homology of smooth $\mathbb Z_2$ orbifold, $\mathcal A_\theta \rtimes \mathbb Z_2$}
In this section we give partial results regarding the Hochschild homology of the smooth non-commutative $\mathbb Z_2$ toroidal orbifold, $\mathcal A_\theta \rtimes \mathbb Z_2$. We also present a  lemma which characterises a class of elements of the group whose dimension is an open problem as I write this article. \par
I also have strong conviction that the method we used to calculate the (co)homology dimensions for the algebraic non-commutative orbifold will be useful and instrumental in computing the $HH_1(\mathcal A_\theta \rtimes \mathbb Z_2)$ whose dimension remains uncalculated. \par

\begin{thm}: For $\theta \notin \mathbb Q$, we have $HH_2(\mathcal A_\theta \rtimes \mathbb Z_2) \cong \mathbb C.$ 
\end{thm}
\begin{proof}
We consider the map $(1 \otimes b_2)$ in the tensored complex. 
\begin{center}
$ 0 \xleftarrow{} {}_{-1}\mathcal A _{\theta} \xleftarrow{b_1} {}_{-1}\mathcal A _{\theta} \oplus {}_{-1}\mathcal A _{\theta} \xleftarrow{b_2} {}_{-1}\mathcal A _{\theta}$\\
\end{center}To calculate the kernel of this map we have a closer look at this map,
\begin{center}
$(1\otimes b_2)(a\otimes I)=a\otimes_{\mathfrak B_\theta}(U_2\otimes I-\lambda\otimes U_2)\otimes e_1-a\otimes_{\mathfrak B_\theta}(\lambda U_1\otimes I-I\otimes U_1)\otimes e_2$.
\end{center}
Using the twisted bimodule structure of ${}_{-1}\mathcal A_\theta$ over $\mathcal A_\theta$, we can simplify the equation to the following,
\begin{center}
$(1\otimes b_2)(a\otimes I)= (\lambda  U_2^{-1}a-aU_2,\lambda a U_1-U_1^{-1} a)$.
\end{center}
Hence we obtain the following relation for an element $(a \otimes 1)$ to reside in $ker( 1\otimes b_2)$. 
\begin{center}
$H_2(\mathcal A_\theta,{}_{-1}\mathcal A_\theta)= \left\lbrace a\in {}_{-1}\mathcal A_\theta | a_{n,m}={\lambda}^{m-1}a_{n-2,m}; a_{n-1,m}=\lambda^n a_{n-1,m-2}\right\rbrace$ .
\end{center}
If an element $\varphi \in {}_{-1}\mathcal A_\theta$ were to satisfy these relations. Then for a fixed $n_0 \in \mathbb Z$ the sequence $(\varphi_{n_0,2m})_m \notin \mathcal S(\mathbb Z)$. This is a contradiction. Hence there are no such nontrivial elements in $\mathcal A_\theta$. Hence we get that:
\begin{center}
$H_2(\mathcal A_\theta,{}_{-1}\mathcal A_\theta)=0$.
\end{center}
\par Now we consider the homology for $g=1$ part in the paracyclic decomposition. We notice through computations similar to algebraic non-commutative orbifold, that, for  $\theta \notin \mathbb Q$, $H_2(\mathcal A_\theta, \mathcal A_\theta) \cong \mathbb C$. Here the generator of the homology group $H_2(\mathcal A_\theta, \mathcal A_\theta)$ is the element $a_{-1,-1}U_1^{-1}U_2^{-1} \in \mathcal A_\theta$, which we see is invariant under the $\mathbb Z_2$ in a similar way as we have demonstrated in earlier calculations pertaining to the algebraic noncommutative orbifold. Hence using the paracyclic decomposition, we have 
\begin{center}$HH_2(\mathcal A_\theta \rtimes \mathbb Z_2) = H_2(\mathcal A_\theta, \mathcal A_\theta)^{\mathbb Z_2} \oplus H_2(\mathcal A_\theta, {}_{-1}\mathcal A_\theta)^{\mathbb Z_2}.$
\end{center}
Using the above formula we conclude that
\begin{center}
$HH_2(\mathcal A_\theta \rtimes \mathbb Z_2) \cong \mathbb C$.
\end{center}
\end{proof}

\begin{thm} For $\theta \notin \mathbb Q$, satisfying the Diophantine condition and $\G \subset SL(2,\mathbb Z)$ a finite group. We have,
\begin{center}
$H_0(\mathcal A_\theta, \mathcal A_\theta)^{\G} \cong \mathbb C$
\end{center}
\end{thm}

\begin{proof}
As we know from [C] that for $\theta \notin \mathbb Q$ satisfying the Diophantine condition, 
\begin{center}
$H_0(\mathcal A_\theta, \mathcal A_\theta) \cong \mathbb C$.
\end{center}
This group is generated by $\varphi_{0,0}$. To compute the invariance we need to deploy the method we used in sections before. We need to push the cycle to the bar complex, $C_\ast(\mathcal A_\theta)$ and then consider the natural action that exists on the bar complex. \newline
Using  the map $h_0: J_0(\mathcal A_\theta) \to C_0(\mathcal A_\theta)$, and the map $k_0: C_(\mathcal A_\theta) \to J_0(\mathcal A_\theta)$, we notice that $h_0=k_0=id$. Hence we observe that the $\mathbb Z_2$ action on the bar complex is induced on to the Kozul complex without any alteration. Hence in the Kozul complex $-1 \cdot \bar{a_{0,0}}=\bar{a_{0,0}}$. Hence we conclude that 
\begin{center}
$H_0(\mathcal A_\theta, \mathcal A_\theta)^{\mathbb Z_2} \cong \mathbb C$
\end{center}
For $\G = \mathbb Z_3, \mathbb Z_4, \text{and } \mathbb Z_6$; we have the same calculation process and hence we conclude that
\begin{center}
$H_0(\mathcal A_\theta, \mathcal A_\theta)^{\G} \cong \mathbb C$
\end{center}

\end{proof}
\begin{thm}
For $\theta \notin \mathbb Q$, satisfying the Diophantine condition and $\G \subset SL(2,\mathbb Z)$ a finite group. We have 
\begin{center}
$H_1(\mathcal A_\theta, \mathcal A_\theta)^{\G} =0$.
\end{center}
\end{thm}
\begin{proof}
Connes proved [C] that for $\theta \notin \mathbb Q$ satisfying the Diophantine condition, 
\begin{center}
$H_1(\mathcal A_\theta, \mathcal A_\theta) \cong \mathbb C^2$.
\end{center}
This group is generated by $\varphi^1_{-1,0}$ and $\varphi^2_{0,-1}$. To compute the invariant sub-space we need to deploy the method we used in previous sections. We need to push the cycle to the bar complex, $C_\ast(\mathcal A_\theta)$ using  the map $h_1: J_1(\mathcal A_\theta) \to C_1(\mathcal A_\theta)$, 
\begin{center}
$h_1(I \otimes e_i) =I \otimes U_j$
\end{center}
and then consider the natural action that exists on the bar complex. Thereafter compare the pull back of the twisted element with the original cycle.\newline


We can represent a general cycle $\varphi \in H_1(\mathcal A_\theta, {}_{-1}\mathcal A_\theta)$ in terms of generators $\bar{\varphi^1_{-1,0}}$ and $\bar{\varphi^2_{0,-1}}$ as follows,
\begin{center}
$\varphi = aU_1^{-1}\otimes e_1 + b U_2^{-1} \otimes e_2$
\end{center}
The twisted pull back of $\varphi$ is the cycle
\begin{center}
$(1 \otimes k_1)(-1 \cdot(1 \otimes h_1)(\varphi)) = (1 \otimes k_1)(-1 \cdot (a U_1^{-1} \otimes U_1 + b U_2 \otimes U_2^{-1}))=(1 \otimes k_1) (a U_1\otimes U_1^{-1} + b U_2 \otimes U_2^{-1})=a U_1 A \otimes e_1 + a U_1 B \otimes e_2 + b U_2 A' \otimes e_1 + b U_2 B' \otimes e_2=(a U_1 A + b U_2 A') \otimes e_1 + (a U_1 B + b U_2 B') \otimes e_2$.
\end{center}
With further simplification using the fact that $\mathcal A_\theta$ is a $\mathfrak B_\theta$ module, we have $U_1 A = -U_1^{-1}$ and $U_2 B'=-U_2^{-1}$. We also notice that $A' =B=0$, hence we have 
\begin{center}
$(1 \otimes k_1)(-1 \cdot(1 \otimes h_1)(aU_1^{-1}\otimes e_1 + b U_2^{-1} \otimes e_2)) =-(aU_1^{-1}\otimes e_1 + b U_2^{-1} \otimes e_2)$.
\end{center}
Whence,
\begin{center}
$H_1(\mathcal A_\theta , \mathcal A_\theta)^{\mathbb Z_2} = 0$.
\end{center}
We can repeat the above proof for $\G =\mathbb Z_3, \mathbb Z_4, \text{and } \mathbb Z_6$, to deduce that
\begin{center}
$H_1(\mathcal A_\theta , \mathcal A_\theta)^{\G} = 0$.
\end{center}
\end{proof}
\begin{conjecture}
For $\theta \notin \mathbb Q$, we conjecture that 
\begin{center}
$HC_{even}(\mathcal A_\theta \rtimes \mathbb Z_2) \cong \mathbb C^6, HC_{odd}(\mathcal A_\theta \rtimes \mathbb Z_2) =0$.
\end{center}
\end{conjecture}
This conjecture when true would mean that the dimension of the $K_0(\mathcal A_\theta \rtimes \mathbb Z_2)$ and $HC_{even}(\mathcal A_\theta \rtimes \mathbb Z_2)$ are same, which is very interesting. Using this we can try for a possible postulation of the Poincare duality.

.


\end{document}